\documentclass[10pt]{amsart}
\usepackage[style=alphabetic]{biblatex}
\usepackage{tikz}
\usepackage{amssymb}
\usepackage{mathtools}
\usepackage{amsthm}
\usepackage{aliascnt}
\usepackage{booktabs,tabularx,array}
\usepackage[all]{xy}
\usepackage{microtype}
\usepackage{xcolor}
\usepackage{listings}
\usepackage{hyperref}
\usepackage[nameinlink]{cleveref}
\usepackage{graphicx}

\makeatletter
    
    \@addtoreset{equation}{section}
  \makeatother

\hypersetup{
 colorlinks,
 linkcolor={teal},
 citecolor={teal},
 urlcolor={teal}
}

\calclayout

\DeclareFieldFormat
  [article,book,inbook,incollection,inproceedings,patent,thesis,unpublished]
  {title}{\emph{#1\isdot}}

\numberwithin{equation}{section}

\theoremstyle{plain}
\newtheorem{theorem}{Theorem}[section]
\newaliascnt{proposition}{theorem}
\newtheorem{proposition}[proposition]{Proposition}
\aliascntresetthe{proposition}
\newaliascnt{lemma}{theorem}
\newtheorem{lemma}[lemma]{Lemma}
\aliascntresetthe{lemma}
\newaliascnt{corollary}{theorem}

\aliascntresetthe{corollary}
\theoremstyle{definition}
\newaliascnt{definition}{theorem}

\aliascntresetthe{definition}
\theoremstyle{remark}
\newaliascnt{remark}{theorem}
\newtheorem{remark}[remark]{Remark}
\aliascntresetthe{remark}

\crefname{theorem}{Theorem}{Theorems}
\crefname{proposition}{Proposition}{Propositions}
\crefname{lemma}{Lemma}{Lemmas}
\crefname{corollary}{Corollary}{Corollaries}
\crefname{definition}{Definition}{Definitions}
\crefname{remark}{Remark}{Remarks}
\crefname{section}{Section}{Sections}
\crefname{equation}{equation}{equations}
\crefname{appendix}{Appendix}{Appendices}

\title[Ap\'ery-type approximations]{Ap\'ery-type approximations and irrationality measures for certain $q$-series}
\author[J. Koizumi]{Junnosuke Koizumi}
\address{RIKEN iTHEMS, Wako, Saitama 351-0198, Japan}
\email{junnosuke.koizumi@riken.jp}
\author[A. Yokoi]{Anju Yokoi}
\address{Ikeda Senior High School Attached to Osaka Kyoiku University, 1-5-1, Midorigaoka, Ikeda-shi, Osaka, 563-0026, Japan}
\email{anju.scorpion@icloud.com}
\date{\today}


\begin{document}

\begin{abstract}
We construct a three-parameter family of rational approximations to values of
$q$-hypergeometric series.  Using these approximations, we prove that, for
every integer $x$ with $|x|\geq2$, the values at $r=x^{-1}$ of Ramanujan's
theta function $\psi(r)=\sum_{n\geq0}r^{n(n+1)/2}$, the generating function
$\Delta(r)=\sum_{m\geq0}d(2m+1)r^m$ of the divisor function restricted to odd integers,
and the generating function $B_4(r)=\sum_{n\geq0}b_4(n)r^n$ for $4$-regular partitions
are irrational.  We further obtain the upper
bounds $18/7$, $18\pi^2/(7\pi^2-24)$, and $3$, respectively, for their
irrationality measures.  We also show that one of the constructed
approximations coincides with the Pad\'e approximation to a Lambert series
due to Coussement--Smet.
\end{abstract}

\maketitle

\setcounter{tocdepth}{1}
\tableofcontents

\section{Introduction}\label{sec:introduction}

\subsection{Background}

We begin by considering, for $q\in\mathbb C$ with $|q|>1$, the Lambert series
\[
\zeta_q(1)=\sum_{n=1}^{\infty}\frac{1}{q^n-1},
\qquad
\operatorname{Ln}_q(2)=\sum_{n=1}^{\infty}
\frac{(-1)^n}{q^n-1}.
\]
Our sign convention for $\operatorname{Ln}_q(2)$ follows
Amdeberhan--Zeilberger~\cite{amdeberhan1998q}.
Erd\H{o}s~\cite{erdos1948irrationality} proved the irrationality of
$\zeta_2(1)$.
Borwein~\cite{borwein1991irrationality} proved that, for every integer $q>1$
and every $0\neq r\in\mathbb Q$ satisfying $q^n+r\neq0$ for all $n\geq1$,
the sum $\sum_{n\geq1}(q^n+r)^{-1}$ is irrational.
Taking $r=-1$ and $r=1$, this result includes both $\zeta_q(1)$ and
$\operatorname{Ln}_q(2)$.
Indeed, separating the even- and odd-indexed terms gives
\[
\operatorname{Ln}_q(2)
=2\zeta_{q^2}(1)-\zeta_q(1)
=-\sum_{n=1}^{\infty}\frac{1}{q^n+1}.
\]

For integers $q>1$, Amdeberhan--Zeilberger gave Ap\'ery-type irrationality
proofs for $\zeta_q(1)$ and $\operatorname{Ln}_q(2)$ based on $q$-WZ pairs,
and showed that the irrationality measure of each is at most $4.8$.
We write
\[
(u;q)_m:=\prod_{j=0}^{m-1}(1-uq^j),
\qquad
(u;q)_0:=1.
\]
The discrete $1$-form they used for $\zeta_q(1)$ is
\[
\omega=F(n,k)\,\delta k+G(n,k)\,\delta n,
\]
\[
F(n,k)=-\frac{(q;q)_k}{(q;q)_{n+k+1}},\qquad
G(n,k)=\frac{q^{n+1}}{q^{n+1}-1}F(n,k),
\]
while the one for $\operatorname{Ln}_q(2)$ is
\[
\omega=F(n,k)\,\delta k+G(n,k)\,\delta n,
\]
\[
F(n,k)=(-1)^k
\frac{(q;q)_k(q;q)_n}{(q;q)_{n+k+1}(-q;q)_n},\qquad
G(n,k)=\frac{q^{n+1}}{q^{n+1}+1}F(n,k).
\]
Here $\delta k$ and $\delta n$ denote unit edges in the $k$- and
$n$-directions, respectively.
In both cases, the relation
\[
F(n+1,k)-F(n,k)=G(n,k+1)-G(n,k)
\]
holds, expressing the fact that the discrete $1$-form $\omega$ is closed.
Consequently, one can choose a potential $c(n,k)$ satisfying the difference
relations
\[
c(n,k+1)-c(n,k)=F(n,k),
\qquad
c(n+1,k)-c(n,k)=G(n,k).
\]
One then chooses a hypergeometric weight $b(n,k)$ and defines
\[
a_n=\sum_{k=0}^nc(n,k)b(n,k),
\qquad
b_n=\sum_{k=0}^nb(n,k).
\tag{1.1}
\]
These two sequences satisfy a common three-term recurrence.
As in Ap\'ery's method~\cite{apery1979irrationalite}, Amdeberhan--Zeilberger used this recurrence and an asymptotic estimate of the Casoratian to show that $a_n/b_n$ converges rapidly to the desired series $\zeta_q(1)$ and $\operatorname{Ln}_q(2)$, thereby proving the irrationality of these values and obtaining estimates of their irrationality measures.

\subsection{The three-parameter construction}

In this paper, we construct a three-parameter family containing the two
$q$-WZ pairs above.
To ensure simultaneously that all denominators appearing below are nonzero,
we assume throughout this section that $|q|>1$, $\alpha\beta\gamma\neq0$, and
\[
\alpha,\ \beta,\ \gamma,\ \frac{\beta q}{\alpha}
\notin q^{-\mathbb Z_{\geq0}}
:=\{1,q^{-1},q^{-2},\ldots\}.
\tag{1.2}
\]
This condition excludes $\beta=1$, $\gamma q^n=1$, and the zeros of the
finite $q$-Pochhammer symbols appearing below.
We also write
\[
\binom nk_q:=\frac{(q;q)_n}{(q;q)_k(q;q)_{n-k}}
\qquad (0\leq k\leq n).
\]

\begin{theorem}\label{thm:construction}
Under condition (1.2), set
\[
\begin{aligned}
\omega_{\alpha,\beta,\gamma}
&=F_{\alpha,\beta,\gamma}(n,k)\,\delta k
+G_{\alpha,\beta,\gamma}(n,k)\,\delta n,\\
F_{\alpha,\beta,\gamma}(n,k)
&=-\frac q\gamma\left(\frac\alpha\gamma\right)^k
\frac{(\beta q/\alpha;q)_k(\alpha;q)_n}
{(\beta;q)_{n+k+1}(\gamma;q)_n},\\
G_{\alpha,\beta,\gamma}(n,k)
&=\frac{\gamma q^n}{\gamma q^n-1}
F_{\alpha,\beta,\gamma}(n,k).
\end{aligned}
\]
Then
\[
F_{\alpha,\beta,\gamma}(n+1,k)-F_{\alpha,\beta,\gamma}(n,k)
=G_{\alpha,\beta,\gamma}(n,k+1)-G_{\alpha,\beta,\gamma}(n,k).
\]
Suppressing the subscripts $\alpha,\beta,\gamma$, define
\[
\begin{aligned}
c(n,k)
&:=\sum_{m=1}^{n}G(m-1,0)
+\sum_{m=1}^{k}F(n,m-1),\\
b(n,k)
&:=q^{k(k+1)/2}
\left(-\frac\gamma\alpha\right)^k
\binom nk_q
\frac{(\beta;q)_{n+k}}
{(\beta q/\alpha;q)_k(\alpha;q)_n}.
\end{aligned}
\]
This $c(n,k)$ satisfies the two difference relations above, and the sequences
$a_n,b_n$ defined by (1.1) satisfy a common three-term recurrence.
Setting $(\alpha,\beta,\gamma)=(q,q,q)$ gives the pair for $\zeta_q(1)$.
Setting $(\alpha,\beta,\gamma)=(q,q,-q)$ gives
$(\alpha/\gamma)^k=(-1)^k$ and hence the pair for
$\operatorname{Ln}_q(2)$.
\end{theorem}

\begin{proposition}\label{prop:target-series}
In addition to the hypotheses of \cref{thm:construction}, suppose that
$|\gamma|>1$.
Then
\[
\mathcal S_{\alpha,\beta,\gamma}(q)
:=\sum_{k=0}^{\infty}F_{\alpha,\beta,\gamma}(0,k)
=\frac{q}{\gamma(\beta-1)}
{}_2\phi_1\!\left(
\begin{matrix}
\alpha\beta^{-1}Q,\ Q\\
\beta^{-1}Q
\end{matrix};Q,\gamma^{-1}
\right),
\qquad Q=q^{-1},
\tag{1.3}
\]
converges absolutely.
\end{proposition}
Here, we use the convention
\[
{}_2\phi_1\!\left(
\begin{matrix}a,b\\c\end{matrix};Q,z
\right)
:=\sum_{k=0}^{\infty}
\frac{(a;Q)_k(b;Q)_k}{(c;Q)_k(Q;Q)_k}z^k.
\]
In Sections 4--6, for three integral specializations of the parameters, we
prove $b_n\neq0$ and
\[
\lim_{n\to\infty}\frac{a_n}{b_n}
=\mathcal S_{\alpha,\beta,\gamma}(q)
\]
in a form adapted to each choice of parameters.
In Section 7, we show that the ratio $a_n/b_n$ arising from the
Coussement--Smet specialization agrees with their Pad\'e approximant.

\subsection{Symmetry}

\begin{proposition}\label{prop:symmetry}
Suppose that condition (1.2) holds for both $(\alpha,\beta,\gamma)$ and
$(\alpha,\gamma,\beta)$, and that $|\beta|,|\gamma|>1$.
Then Heine's second transformation gives
\[
\mathcal S_{\alpha,\beta,\gamma}(q)
=\mathcal S_{\alpha,\gamma,\beta}(q).
\tag{1.4}
\]

Next, let $\beta=q$ and define
\[
(\alpha,\gamma)
\longmapsto
(\alpha^*,\gamma^*)
:=\left(\frac{q^2}{\alpha},\frac{\gamma q}{\alpha}\right).
\]
If both the original and the starred parameters satisfy (1.2), and if
$|\gamma|>1$ and $|\gamma q/\alpha|>1$, then
\[
\mathcal S_{\alpha^*,q,\gamma^*}(q)
=\frac{\alpha\,\mathcal S_{\alpha,q,\gamma}(q)}
{q+(q-\alpha)\mathcal S_{\alpha,q,\gamma}(q)}.
\tag{1.5}
\]
Applying this transformation twice returns $(\alpha,\gamma)$.
Moreover, if $a_n,b_n$ denote the sequences associated with the original
parameters and $a_n^*,b_n^*$ those associated with the starred parameters,
then
\[
b_n+\left(1-\frac\alpha q\right)a_n
=\frac{(\gamma q/\alpha;q)_n}{(\gamma;q)_n}b_n^*.
\tag{1.6}
\]
\end{proposition}
Thus, for rational parameters, the values approximated by the two
constructions are related by a linear fractional
transformation with rational coefficients; consequently, irrationality and
the irrationality measure are preserved.

In this paper, we use the symmetry relation (1.6) satisfied by the rational
approximations to sharpen the denominator estimates, thereby obtaining improved
upper bounds for the irrationality measures. This method goes back at least to
Hata's 1995 work \cite{hata1995beukers}, and was subsequently developed by Rhin
and Viola and by Zudilin
\cite{rhinviola1996permutation,rhinviola2001group,zudilin2002qzeta,zudilin2004heine,zudilin2004arithmetic}.

\subsection{Main irrationality results}

For a real irrational number $\xi$, define its irrationality measure by
\[
\mu(\xi)
:=\sup\left\{\mu>0:
0<\left|\xi-\frac PQ\right|<Q^{-\mu}
\text{ for infinitely many }(P,Q)\in\mathbb Z\times\mathbb Z_{>0}
\right\}.
\]

Let $d(n)$ denote the number of positive divisors of a positive integer $n$.
For a nonnegative integer $n$, let $b_4(n)$ denote the number of $4$-regular
partitions of $n$, that is, partitions using no part divisible by $4$, with
$b_4(0)=1$ corresponding to the empty partition.
For $|r|<1$, define the three functions studied in this paper by
\[
\begin{aligned}
\psi(r)
&:=\sum_{n=0}^{\infty}r^{n(n+1)/2},\\
\Delta(r)
&:=\sum_{m=0}^{\infty}d(2m+1)r^m
=\sum_{k=0}^{\infty}\frac{r^k}{1-r^{2k+1}},\\
B_4(r)
&:=\sum_{n=0}^{\infty}b_4(n)r^n
=\frac{(r^4;r^4)_\infty}{(r;r)_\infty}.
\end{aligned}
\]
The equality between the two series representations of $\Delta$
follows because the coefficient of $r^m$ on the right counts the
factorizations $2m+1=(2j+1)(2k+1)$.

\begin{theorem}\label{thm:main}
For every integer $x$ with $|x|\geq2$, the numbers $\psi(x^{-1})$,
$\Delta(x^{-1})$, and $B_4(x^{-1})$ are irrational, and
\[
\begin{aligned}
\mu\!\left(\psi(x^{-1})\right)
&\leq\frac{18}{7}=2.571428\ldots,\\
\mu\!\left(\Delta(x^{-1})\right)
&\leq\frac{18\pi^2}{7\pi^2-24}=3.940203\ldots,\\
\mu\!\left(B_4(x^{-1})\right)
&\leq3.
\end{aligned}
\]
\end{theorem}

These values are obtained, respectively, by the specializations $(q,\alpha,\beta,\gamma)=(x^2,x,x^2,x)$, $(x^2,x^2,x,x)$, and $(x^2,-x,x^2,x)$:
\[
\mathcal S_{x,x^2,x}(x^2)=\frac{\psi(x^{-1})-1}{1-x^{-1}},\qquad
\mathcal S_{x^2,x,x}(x^2)=\Delta(x^{-1}),\qquad
\mathcal S_{-x,x^2,x}(x^2)=\frac{B_4(x^{-1})-1}{1+x^{-1}}.
\]
The result for $\psi(r)$ improves the bound obtained from a result of Bundschuh~\cite{bundschuh1974tschakaloff}:
\[
\mu\!\left(\psi(x^{-1})\right)
\leq\frac{3+\sqrt5}{2}
=2.618033\ldots.
\]
The result for $\Delta(r)$ improves the bound obtained from a result of Bundschuh--Zudilin~\cite{bundschuhzudilin2008rational}:
\[
\mu\!\left(\Delta(x^{-1})\right)\leq6.
\]
Finally, the result for $B_4$ improves the bound obtained from a result of 
Rochev~\cite{rochev2010linear}:
\[
\mu\!\left(B_4(x^{-1})\right)
\leq3+\sqrt6
=5.449489\ldots.
\]
Nesterenko's theorem~\cite{nesterenko1996modular} implies that
$\psi(x^{-1})$ and $B_4(x^{-1})$ are in fact transcendental for every
integer $x$ with $|x|\geq2$.

\begin{remark}
    For an integer $x\geq 2$, the special value $\psi(x^{-1})$ of Ramanujan's theta function is equal to the number obtained by concatenating $1,10,100,1000,\dots$ in order in base $x$:
    \[
    \psi(x^{-1})=\bigl(1.101001000100001\ldots\bigr)_{(x)}.
    \]
    Our \cref{thm:main} shows that the irrationality measure of this number is at most $18/7$.
\end{remark}
\subsection{The Coussement--Smet approximants}

Coussement--Smet~\cite{coussementsmet2009irrationality} constructed Pad\'e
approximants to the series
\[
h^\pm(q_1,q_2)
:=\sum_{m=1}^{\infty}\frac{q_1^m}{1\pm q_2^m}
\]
for $q_1\in(0,1)\cap\mathbb Q$ and $q_2=p_2^{-1}$, where
$p_2\in\mathbb Z_{\geq2}$.
In particular, for an integer $p>1$ and relatively prime positive integers
$r_1,r_2$, setting $q_1=p^{-r_1}$ and $q_2=p^{-r_2}$ yields
\[
\mu\!\left(h^-(p^{-r_1},p^{-r_2})\right)
\leq m^-(r_2).
\]
Here $m^-(r_2)$ is the explicit constant defined in equation (1.14) of their
paper and depends only on the second exponent $r_2$ after reduction.

Set $p_1:=p^{r_1}$ and $p_2:=p^{r_2}$.
Specializing the three-parameter construction to
$(q,\alpha,\beta,\gamma)=(p_2,p_2,p_1,p_2)$
gives
\[
\mathcal S_{p_2,p_1,p_2}(p_2)
=h^-(p_1^{-1},p_2^{-1}).
\]
Let $R_{n,N}$ denote the approximant for which the degree of the
Coussement--Smet Pad\'e polynomial is $n$ and the index of the evaluation
point is $N$.
In Section 7, we prove that
\[
\frac{a_n}{b_n}=R_{n,n+1}.
\]
Thus, when the indices are matched by taking $N=n+1$, the rational
approximants before clearing denominators coincide.
Multiplying the numerator and denominator of this fraction by the factor used
by Coussement--Smet and applying their asymptotic estimates also recovers the
irrationality-measure bound above.

\subsection{Summary of results}

In the first three rows we assume that $x$ is an integer with $|x|\geq2$;
in the last row we assume that $p>1$ is an integer and $r_1,r_2$ are
relatively prime positive integers.

\begin{center}
%\small
\setlength{\tabcolsep}{2pt}
\begin{tabularx}{\textwidth}{@{}>{\raggedright\arraybackslash}X >{\raggedright\arraybackslash}X >{\raggedright\arraybackslash}X >{\raggedright\arraybackslash}X >{\raggedright\arraybackslash}X@{}}
\toprule
$(q,\alpha,\beta,\gamma)$ & Target & Previous bound & Bound in this paper & Comparison \\
\midrule
$(x^2,x,x^2,x)$ & $\psi(x^{-1})$ & $(3+\sqrt5)/2$ \cite{bundschuh1974tschakaloff} & $18/7$ & Improved \\
$(x^2,x^2,x,x)$ & $\Delta(x^{-1})$ & $6$ \cite{bundschuhzudilin2008rational} & $18\pi^2/(7\pi^2-24)$ & Improved \\
$(x^2,-x,x^2,x)$ & $B_4(x^{-1})$ & $3+\sqrt6$ \cite{rochev2010linear} & $3$ & Improved \\
$(p^{r_2},p^{r_2},p^{r_1},p^{r_2})$ & $h^-(p^{-r_1},p^{-r_2})$ & $m^-(r_2)$~\cite{coussementsmet2009irrationality} & $m^-(r_2)$ & Recovered \\
\bottomrule
\end{tabularx}
\end{center}

In \cref{sec:construction}, we prove
\cref{thm:construction,prop:target-series,prop:symmetry}.
\cref{sec:tools} prepares a general lemma for estimating irrationality measures,
and \cref{sec:psi,sec:delta,sec:b4} treat in turn the three specializations in
\cref{thm:main}.
In \cref{sec:coussement-smet}, we prove the identification with the Coussement--Smet Pad\'e
approximants.
\subsection*{Acknowledgement}
The authors would like to thank Professor Wadim Zudilin for bringing several relevant references to their attention and for his helpful comments on the previous results related to this work.
\subsection*{Human--AI collaboration}
This research originated in the second author's discovery of a
two-parameter $q$-WZ pair and was subsequently developed through extensive
interaction between the authors and OpenAI's GPT-5.6 Sol.
Several key ideas used in the proof of \cref{thm:main} emerged from these
interactions.
We regard this work as an instance of human--AI collaboration, rather than
autonomous research conducted by AI: the authors directed the investigation,
evaluated the AI-generated suggestions, and independently verified all
mathematical arguments and references.
The authors reviewed and edited all AI-assisted content and take full
responsibility for the paper.

\section{Construction of a three-parameter \texorpdfstring{$q$-WZ pair}{q-WZ pair}}\label{sec:construction}

In this section, we construct a three-parameter family containing the two $q$-WZ pairs of Amdeberhan--Zeilberger.
By combining a potential obtained from a closed discrete $1$-form with a hypergeometric weight, we establish a common second-order recurrence and its relation to the series being approximated.

\subsection{Definitions and assumptions}

As in \cref{sec:introduction}, assume that $|q|>1$, $\alpha\beta\gamma\neq0$, and
\[
\alpha,\ \beta,\ \gamma,\ \frac{\beta q}{\alpha}
\notin q^{-\mathbb Z_{\geq0}}.
\]
Under these conditions, the finite $q$-products and the factors $\gamma q^n-1$ appearing below do not vanish.
Set
\[
\begin{aligned}
\omega_{\alpha,\beta,\gamma}
&=F_{\alpha,\beta,\gamma}(n,k)\,\delta k
+G_{\alpha,\beta,\gamma}(n,k)\,\delta n,\\
F_{\alpha,\beta,\gamma}(n,k)
&=-\frac q\gamma\left(\frac\alpha\gamma\right)^k
\frac{(\beta q/\alpha;q)_k(\alpha;q)_n}
{(\beta;q)_{n+k+1}(\gamma;q)_n},\\
G_{\alpha,\beta,\gamma}(n,k)
&=\frac{\gamma q^n}{\gamma q^n-1}
F_{\alpha,\beta,\gamma}(n,k).
\end{aligned}
\]
The $q$-WZ pairs of Amdeberhan--Zeilberger for $\zeta_q(1)$ and $\operatorname{Ln}_q(2)$ correspond to $(\alpha,\beta,\gamma)=(q,q,q)$ and $(q,q,-q)$, respectively.
In what follows, we omit the subscripts $\alpha,\beta,\gamma$ and write $\omega$, $F(n,k)$, and $G(n,k)$.

\begin{lemma}
    The discrete $1$-form $\omega$ is closed:
    \[
    F(n+1,k)-F(n,k)=G(n,k+1)-G(n,k).
    \]
\end{lemma}

\begin{proof}
Introduce the three variables
\[
X=\gamma q^n,
\qquad
Y=\alpha q^n,
\qquad
Z=\alpha^{-1}\beta q^{k+1}.
\]
The definitions give
\[
\frac{F(n+1,k)}{F(n,k)}
=\frac{1-Y}{(1-X)(1-YZ)},
\]
\[
\frac{F(n,k+1)}{F(n,k)}
=\frac YX\frac{1-Z}{1-YZ},
\qquad
\frac{G(n,k)}{F(n,k)}
=\frac X{X-1}.
\]
Therefore
\[
\begin{aligned}
&\frac{F(n+1,k)-F(n,k)-G(n,k+1)+G(n,k)}{F(n,k)}\\
&=\frac{1-Y}{(1-X)(1-YZ)}
-1
+\frac{Y(1-Z)}{(1-X)(1-YZ)}
-\frac X{1-X}
=0,
\end{aligned}
\]
and hence
$F(n+1,k)-F(n,k)=G(n,k+1)-G(n,k)$
for all $n,k\geq 0$.
\end{proof}

\subsection{The weight and potential}

We define a weight $b(n,k)$ by
\[
b(n,k)
=q^{k(k+1)/2}
\left(-\frac\gamma\alpha\right)^k
\binom nk_q
\frac{(\beta;q)_{n+k}}
{(\beta q/\alpha;q)_k(\alpha;q)_n}.
\]
If $\varepsilon\in\{1,-1\}$ and $(\alpha,\beta,\gamma)=(q,q,\varepsilon q)$, then
\[
b(n,k)
=(-\varepsilon)^kq^{k(k+1)/2}
\frac{(q;q)_{n+k}}
{(q;q)_k^2(q;q)_{n-k}},
\]
which gives the two weights used by Amdeberhan--Zeilberger.

With the convention that an empty sum is zero, define a potential of $\omega$ by
\[
c(n,k)
:=\sum_{m=1}^{n}G(m-1,0)
+\sum_{m=1}^{k}F(n,m-1).
\]
In particular, $c(0,0)=0$.
Substitution of the individual terms gives
\[
c(n,k)
=\sum_{m=1}^{n}
\frac{q^m(\alpha;q)_{m-1}}
{(\beta;q)_m(\gamma;q)_m}
-\frac q\gamma\sum_{m=1}^{k}
\left(\frac\alpha\gamma\right)^{m-1}
\frac{(\beta q/\alpha;q)_{m-1}(\alpha;q)_n}
{(\beta;q)_{n+m}(\gamma;q)_n}.
\]
The definition and the relation proved above imply
\[
c(n,k+1)-c(n,k)=F(n,k),
\qquad
c(n+1,k)-c(n,k)=G(n,k).
\]
Define the sequences $a_n,b_n$ by
\[
a_n
:=\sum_{k=0}^{n}c(n,k)b(n,k),
\qquad
b_n
:=\sum_{k=0}^{n}b(n,k).
\]

\subsection{A common second-order recurrence}

\begin{proposition}\label{prop:common-recurrence}
When $q,\alpha,\beta,\gamma$ are treated as indeterminates, $a_n$ and $b_n$ satisfy the common second-order recurrence
\[
u_{n+2}+y_1(n)u_{n+1}+y_0(n)u_n=0,
\]
where $y_0(n),y_1(n)$ are rational functions of $q^n,q,\alpha,\beta,\gamma$.
\end{proposition}

\begin{proof}
Define the shift operator $N$ by $Nf(n,k)=f(n+1,k)$, with coefficients acting from the left.
Applying \texttt{QDifferenceEquations:-Zeilberger} in Maple 2026~\cite{maple2026} to $b(n,k)$ and dividing the entire operator by the coefficient of $N^2$ yields
\[
L=y_0(n)+y_1(n)N+N^2
\]
and a term $T_b(n,k)$ satisfying
\[
Lb(n,k)=T_b(n,k+1)-T_b(n,k)
\]
as a rational function of $q^n,q^k,q,\alpha,\beta,\gamma$.
Here, we use the same notation $L$ for sequences by setting $Lu_n:=y_0(n)u_n+y_1(n)u_{n+1}+u_{n+2}$.
The coefficient $y_0(n)$, which will be used later, is given by
\[
\begin{aligned}
\mathcal K_n
:={}&\beta\gamma q^{3n+2}
-(\alpha+\beta+\gamma)q^{n+1}
+\alpha+q,\\
y_0(n)
={}&q\,
\frac{(q^{n+1}-1)(\beta q^n-1)(\gamma q^{n+1}-\alpha)\mathcal K_{n+1}}
{(\alpha q^{n+1}-1)(\beta q^{n+2}-\alpha)
(\gamma q^{n+1}-1)\mathcal K_n}.
\end{aligned}
\]
When $L$ is applied to the product $c(n,k)b(n,k)$ and the two difference relations for the potential are used, the remaining term is
\[
\begin{aligned}
H(n,k)
:={}&y_1(n)b(n+1,k)G(n,k)\\
&+b(n+2,k)\bigl(G(n,k)+G(n+1,k)\bigr)\\
&-F(n,k-1)T_b(n,k).
\end{aligned}
\]
For $k=0$, the last product is understood to be zero.
Applying the same algorithm to $H(n,k)$ gives a relation of order $0$ in the $n$-shift, with coefficient $1$,
\[
H(n,k)=T_a(n,k+1)-T_a(n,k).
\]
Consequently,
\[
\begin{aligned}
L\bigl(c(n,k)b(n,k)\bigr)
={}&\bigl(c(n,k)T_b(n,k+1)+T_a(n,k+1)\bigr)\\
&-\bigl(c(n,k-1)T_b(n,k)+T_a(n,k)\bigr).
\end{aligned}
\]
The two terms occurring on the right satisfy
\[
T_b(n,0)=T_a(n,0)=T_b(n,n+3)=T_a(n,n+3)=0.
\]
Extending $b(n,k)$ and $c(n,k)b(n,k)$ by zero for $k<0$ or $k>n$ and summing the last two identities over $k=0,\ldots,n+2$, we obtain
\[
Lb_n=0,
\qquad
La_n=0.
\]

The expanded expressions for $y_1,T_b,T_a$ and the identities above can be generated by Maple from the same input.
The computation symbolically verifies the two identities, the displayed expression for $y_0(n)$, the order of the operator, the nonvanishing of the coefficients, and the vanishing of the four terms at $k=0,n+3$.
Archival details for the complete input and saved output are given in
\cref{app:maple}.
\end{proof}

\subsection{The target series}

Setting $n=0$ gives
\[
F(0,k)
=-\frac q\gamma\left(\frac\alpha\gamma\right)^k
\frac{(\beta q/\alpha;q)_k}{(\beta;q)_{k+1}}.
\]
The ratio of consecutive terms is
\[
\frac{F(0,k+1)}{F(0,k)}
=\frac\alpha\gamma
\frac{1-\beta q^{k+1}/\alpha}
{1-\beta q^{k+1}}
\longrightarrow\frac1\gamma.
\]
Thus, when $|\gamma|>1$,
\[
\mathcal S_{\alpha,\beta,\gamma}(q)
:=\sum_{k=0}^{\infty}F(0,k)
\]
converges absolutely.
Put $Q=q^{-1}$ and use
\[
(u;q)_k=(-u)^kq^{k(k-1)/2}(u^{-1};Q)_k
\]
to obtain
\[
F(0,k)
=\frac{q}{\gamma(\beta-1)}
\frac{(\alpha\beta^{-1}Q;Q)_k}
{(\beta^{-1}Q;Q)_k}\gamma^{-k}.
\]
Inserting $(Q;Q)_k/(Q;Q)_k$ in each term and using the convention defined in Section~1 gives
\[
\mathcal S_{\alpha,\beta,\gamma}(q)
=\frac{q}{\gamma(\beta-1)}
{}_2\phi_1\!\left(
\begin{matrix}
\alpha\beta^{-1}Q,\ Q\\
\beta^{-1}Q
\end{matrix};Q,\gamma^{-1}
\right).
\]

Since the discrete $1$-form $\omega$ is closed, we have, for $N,K>0$,
\[
\sum_{j=0}^{K-1}F(0,j)
+\sum_{m=0}^{N-1}G(m,K)
=\sum_{j=0}^{K-1}F(N,j)
+\sum_{m=0}^{N-1}G(m,0).
\]
By our assumptions for the parameters, $(u^{-1};Q)_r$ is bounded and bounded away from zero as $r\geq0$ varies for every $u\in\{\alpha,\beta,\gamma,\beta q/\alpha\}$.
Hence there exist $C,A>0$, depending only on the parameters, such that
\[
|G(m,K)|
\leq C|\gamma|^{-K}A^m
|q|^{-m(m+1)/2-mK}.
\]
Applying the same transformation formula to $F(N,j)$, for fixed $K$ we have
\[
\sum_{j=0}^{K-1}F(N,j)\longrightarrow0
\qquad(N\to\infty),
\]
while, under $|\gamma|>1$,
\[
\sum_{m=0}^{\infty}|G(m,K)|=O\!\left(|\gamma|^{-K}\right)
\qquad(K\to\infty).
\]
Letting first $N\to\infty$ and then $K\to\infty$, we obtain
\[
\mathcal S_{\alpha,\beta,\gamma}(q)
=\sum_{m=0}^{\infty}G(m,0),
\qquad
G(m,0)
=\frac{q^{m+1}(\alpha;q)_m}
{(\beta;q)_{m+1}(\gamma;q)_{m+1}}.
\]
Taking $(N,K)=(n,k)$ in the same identity gives, for all $n,k\geq0$,
\[
c(n,k)
=\sum_{j=0}^{k-1}F(0,j)
+\sum_{m=0}^{n-1}G(m,k),
\]
and therefore
\[
\mathcal S_{\alpha,\beta,\gamma}(q)-c(n,k)
=\sum_{m=n}^{\infty}G(m,k).
\]
In particular,
\[
\mathcal S_{\alpha,\beta,\gamma}(q)b_n-a_n
=\sum_{k=0}^{n}b(n,k)
\sum_{m=n}^{\infty}G(m,k).
\]
Convergence of the quotient $a_n/b_n$ requires, in addition, that $b_n\neq0$ and an estimate for the size of the right-hand side.
In Sections~4--6, we prove these facts separately for three integer specializations of the parameters.
In Section~7, we show that the resulting specialized quotients $a_n/b_n$ agree with known Pad\'e approximants.

\subsection{Symmetry in \texorpdfstring{$\beta$ and $\gamma$}{beta and gamma}}

Assume that the conditions hold for both $(\alpha,\beta,\gamma)$ and $(\alpha,\gamma,\beta)$, and that $|\beta|,|\gamma|>1$.
Since the right-hand side of the representation
\[
\mathcal S_{\alpha,\beta,\gamma}(q)
=\sum_{m=0}^{\infty}
\frac{q^{m+1}(\alpha;q)_m}
{(\beta;q)_{m+1}(\gamma;q)_{m+1}}
\]
obtained above is symmetric in $\beta$ and $\gamma$, we have
\[
\mathcal S_{\alpha,\beta,\gamma}(q)
=\mathcal S_{\alpha,\gamma,\beta}(q).
\]
The same identity also follows from Heine's second transformation formula.
Indeed, under $|Q|<1$ and $|\beta^{-1}|,|\gamma^{-1}|<1$,
\[
\mathcal S_{\alpha,\beta,\gamma}(q)
=\frac{q}{\gamma(\beta-1)}
\frac{(\beta^{-1};Q)_\infty(\gamma^{-1}Q;Q)_\infty}
{(\beta^{-1}Q;Q)_\infty(\gamma^{-1};Q)_\infty}
{}_2\phi_1\!\left(
\begin{matrix}
\alpha\gamma^{-1}Q,\ Q\\
\gamma^{-1}Q
\end{matrix};Q,\beta^{-1}
\right).
\]
Since
\[
\frac{(\beta^{-1};Q)_\infty}{(\beta^{-1}Q;Q)_\infty}
\frac{(\gamma^{-1}Q;Q)_\infty}{(\gamma^{-1};Q)_\infty}
=\frac{\beta-1}{\beta}\frac{\gamma}{\gamma-1},
\]
the right-hand side equals $\mathcal S_{\alpha,\gamma,\beta}(q)$.

\subsection{The case \texorpdfstring{$\alpha=q$}{alpha=q}}

Bundschuh--Zudilin~\cite{bundschuhzudilin2008rational} defined the generalized $q$-logarithm by
\[
\ell_q(u,v)
:=u\sum_{n=1}^{\infty}
\frac{v^n}{q^n-u}.
\]
Setting $\alpha=q$ and assuming $|\gamma|>1$, we obtain
\[
\mathcal S_{q,\beta,\gamma}(q)
=\sum_{k=0}^{\infty}
\frac{(q/\gamma)^{k+1}}
{\beta q^k-1}
=\ell_q\!\left(\frac q\beta,\frac q\gamma\right).
\]
If, in addition, $|\beta|>1$, then both series converge absolutely, and the result of the preceding subsection gives
\[
\ell_q\!\left(\frac q\beta,\frac q\gamma\right)
=\ell_q\!\left(\frac q\gamma,\frac q\beta\right).
\]

\subsection{The case \texorpdfstring{$\beta=q$}{beta=q}}

First assume that $\alpha\neq q$.
Substituting $a=\alpha Q$ and $z=\gamma^{-1}$ into the $q$-binomial theorem
\[
\sum_{k=0}^{\infty}
\frac{(a;Q)_k}{(Q;Q)_k}z^k
=\frac{(az;Q)_\infty}{(z;Q)_\infty}
\qquad(|z|<1)
\]
gives
\[
\mathcal S_{\alpha,q,\gamma}(q)
=\frac{P_{\alpha,\gamma}(q)-1}{1-\alpha/q},
\qquad
P_{\alpha,\gamma}(q)
:=\frac{(\alpha\gamma^{-1}Q;Q)_\infty}
{(\gamma^{-1};Q)_\infty}.
\]
In particular, if $q,\alpha,\gamma\in\mathbb Q$ and $\alpha\neq q$, this linear relation shows that one of $\mathcal S_{\alpha,q,\gamma}(q)$ and $P_{\alpha,\gamma}(q)$ is irrational if and only if the other is, and that their irrationality measures then coincide.

\subsection{A parameter transformation in the case \texorpdfstring{$\beta=q$}{beta=q}}\label{sec:reciprocal-product-duality}

Define
\[
(\alpha,\gamma)
\longmapsto
(\alpha^*,\gamma^*)
:=\left(\frac{q^2}{\alpha},\frac{\gamma q}{\alpha}\right).
\]
Applying this transformation twice returns $(\alpha,\gamma)$.
Assume that both the original and starred parameters satisfy the conditions of this section, and that $|\gamma|>1$ and $|\gamma q/\alpha|>1$.
If $\alpha\neq q$, the definition of the infinite product gives
\[
P_{\alpha^*,\gamma^*}(q)
=P_{\alpha,\gamma}(q)^{-1},
\]
and hence
\[
\mathcal S_{\alpha^*,q,\gamma^*}(q)
=\frac{(\alpha/q)\,\mathcal S_{\alpha,q,\gamma}(q)}
{1+(1-\alpha/q)\mathcal S_{\alpha,q,\gamma}(q)}.
\]
When $\alpha=q$, this formula is the identity relating the same parameters and series.

In fact, this relation holds at the level of finite approximations.
Write $a_n,b_n$ for the sequences obtained from the original parameters and $a_n^*,b_n^*$ for those obtained from the starred parameters, and set
\[
R_n:=\frac{(\gamma q/\alpha;q)_n}{(\gamma;q)_n}.
\]

\begin{lemma}
Let $\beta=q$.
Then, for every $n\geq0$,
\[
\frac\alpha q a_n=R_na_n^*,
\qquad
b_n+\left(1-\frac\alpha q\right)a_n=R_nb_n^*.
\tag{2.1}
\]
\end{lemma}

\begin{proof}
Let $y_i(n)$ and $y_i^*(n)$ denote the recurrence coefficients for the two sets of parameters.
Then
\[
y_1^*(n)=y_1(n)\frac{R_{n+1}}{R_{n+2}},
\qquad
y_0^*(n)=y_0(n)\frac{R_n}{R_{n+2}}.
\]
Consequently, $R_na_n^*$ and $R_nb_n^*$ satisfy the same recurrence as $a_n,b_n$.
The initial values are
\[
a_0=0,
\qquad
b_0=1,
\]
\[
\begin{aligned}
a_1
&=-\frac{q(q-1)(-\alpha+\gamma q^2+\gamma q-q)}
{(q^2-\alpha)(\alpha-1)(\gamma-1)},\\
b_1
&=-\frac{(q-1)(\alpha+\gamma q^3-\gamma q-q^2)}
{(q^2-\alpha)(\alpha-1)}.
\end{aligned}
\]
The starred initial values are obtained by substituting $(\alpha^*,\gamma^*)$ into these two formulas, and both identities in (2.1) hold for $n=0,1$.
The second-order recurrence therefore proves (2.1).
\end{proof}

The transformation of the coefficients and the four initial values can be verified in Maple.
More precisely, the difference between the two sides of each identity simplifies to zero as a rational function.
Archival details for the complete input and saved output are given in
\cref{app:maple-duality}.

In particular, whenever the denominator is nonzero,
\[
\frac{a_n^*}{b_n^*}
=\frac{(\alpha/q)a_n}
{b_n+(1-\alpha/q)a_n},
\]
so the fractional linear transformation of the series already holds at every finite stage.

\section{Tools for irrationality proofs}\label{sec:tools}

In this section, we develop the common tools that will be used in the irrationality proofs in the subsequent sections.

\subsection{A recurrence for the Casoratian}

\begin{lemma}\label{lem:casoratian}
Suppose that the coefficients $y_0(n),y_1(n)$ are defined for $n\geq0$ and that the complex sequences $(a_n)$ and $(b_n)$ satisfy the same recurrence
\[
u_{n+2}+y_1(n)u_{n+1}+y_0(n)u_n=0.
\]
If $W_n:=a_nb_{n+1}-a_{n+1}b_n$, then
\[
W_{n+1}=y_0(n)W_n
\]
for every $n\geq0$, and
\[
W_n=W_0\prod_{j=0}^{n-1}y_0(j).
\]
In particular, if $W_0\neq0$ and $y_0(j)\neq0$ for every $j\geq0$, then $W_n\neq0$ for every $n\geq0$.
\end{lemma}

\begin{proof}
Substitution of the recurrence gives
\[
\begin{aligned}
W_{n+1}
&=a_{n+1}b_{n+2}-a_{n+2}b_{n+1}\\
&=a_{n+1}\bigl(-y_1(n)b_{n+1}-y_0(n)b_n\bigr)\\
&\quad-\bigl(-y_1(n)a_{n+1}-y_0(n)a_n\bigr)b_{n+1}\\
&=y_0(n)W_n.
\end{aligned}
\]
The product formula follows by iterating this identity.
\end{proof}

If rational sequences $a_n,b_n$ are multiplied by a nonzero common denominator $M_n$ so that
\[
A_n:=M_na_n\in\mathbb Z,
\qquad
B_n:=M_nb_n\in\mathbb Z,
\]
then
\[
A_nB_{n+1}-A_{n+1}B_n
=M_nM_{n+1}W_n.
\]
Thus, if $W_n\neq0$, then the corresponding determinant obtained after multiplying $a_n,b_n$ by a common denominator is also nonzero.

\subsection{Estimates for finite products of polynomials}

\begin{lemma}\label{lem:finite-products}
Fix an integer $x$ and put $\rho:=|x|\geq2$.
There exist constants $c_\rho,C_\rho>0$, depending only on $\rho$, such that
\[
c_\rho
\leq\prod_{e\in E}|1-x^{-e}|
\leq C_\rho
\]
for every finite set $E$ of positive integers.
\end{lemma}

\begin{proof}
For every $e\geq1$, we have
$1-\rho^{-e}\leq|1-x^{-e}|\leq1+\rho^{-e}$.
The two infinite products
\[
\prod_{e=1}^{\infty}(1-\rho^{-e}),
\qquad
\prod_{e=1}^{\infty}(1+\rho^{-e})
\]
converge to a positive number and a finite number, respectively, which proves the assertion.
\end{proof}

\begin{lemma}\label{lem:cyclotomic-growth}
Let $\Phi_d(X)$ be the $d$th cyclotomic polynomial.
Fix an integer $x$, put $\rho:=|x|\geq2$, and fix a constant $M>0$.
Suppose that the nonnegative integers $m_{d,N}$ satisfy
\[
m_{d,N}\leq\frac{MN}{d}
\qquad(1\leq d\leq N).
\]
Then
\[
\log\left|
\prod_{d=1}^{N}\Phi_d(x)^{m_{d,N}}
\right|
=\left(\sum_{d=1}^{N}m_{d,N}\varphi(d)\right)\log\rho
+O_{\rho,M}(N\log N).
\]
\end{lemma}

\begin{proof}
Write $\mu$ for the M\"obius function.
Expressing the cyclotomic polynomial as a product, we obtain
\[
\log|\Phi_d(x)|-\varphi(d)\log\rho
=\sum_{e\mid d}\mu(d/e)\log|1-x^{-e}|.
\]
Since $\bigl|\log|1-x^{-e}|\bigr|=O_\rho(\rho^{-e})$, the sum of the absolute values of the errors is bounded by
\[
\sum_{d=1}^{N}m_{d,N}
\sum_{e\mid d}\bigl|\log|1-x^{-e}|\bigr|
\leq MN
\sum_{eh\leq N}
\frac{\bigl|\log|1-x^{-e}|\bigr|}{eh}
=O_{\rho,M}(N\log N).
\]
This proves the claim.
\end{proof}

\subsection{A lemma for estimating the irrationality measure}

\begin{lemma}\label{lem:irrationality-measure}
Let $\xi\in\mathbb R$, and suppose that pairs of integers $(A_n,B_n)$ satisfy
\[
\lvert A_n\rvert+\lvert B_n\rvert
\leq X^{\kappa n^2+o(n^2)},
\qquad
0<\lvert B_n\xi-A_n\rvert
\leq X^{-\lambda n^2+o(n^2)}
\]
for some $X>1$ and $\kappa,\lambda>0$.
Then $\xi$ is irrational.
Moreover, if $B_n\neq0$ and
\[
A_nB_{n+1}-A_{n+1}B_n\neq0
\]
for every sufficiently large $n$, then the irrationality measure of $\xi$ satisfies
\[
\mu(\xi)\leq1+\frac{\kappa}{\lambda}.
\]
\end{lemma}

\begin{proof}
Suppose that $\xi=p/q\in\mathbb Q$, where $p\in\mathbb Z$ and $q\in\mathbb Z_{>0}$.
By assumption, the error tends to zero while remaining nonzero.
On the other hand,
\[
q\lvert B_n\xi-A_n\rvert
=\lvert pB_n-qA_n\rvert
\]
is a positive integer and hence is at least $1$, a contradiction.
Therefore $\xi$ is irrational.

Fix $\eta>0$.
Choose $\delta>0$ sufficiently small that
\[
\frac{\kappa+\delta}{\lambda-\delta}
<\frac{\kappa}{\lambda}+\frac\eta2.
\]
For every sufficiently large $n$,
\[
\lvert A_n\rvert+\lvert B_n\rvert
\leq X^{(\kappa+\delta)n^2},
\qquad
\lvert B_n\xi-A_n\rvert
\leq X^{-(\lambda-\delta)n^2}.
\]
Given $a\in\mathbb Z$ and a sufficiently large $b\in\mathbb Z_{>0}$, choose the smallest nonnegative integer $n$ such that
\[
X^{(\lambda-\delta)n^2}\geq2b.
\]
For sufficiently large $b$, both $n$ and $n+1$ lie in the range where the above estimates hold and both linear forms are nonzero.
For each $j\in\{n,n+1\}$,
\[
\lvert B_j\xi-A_j\rvert
\leq X^{-(\lambda-\delta)j^2}
\leq X^{-(\lambda-\delta)n^2}
\leq\frac1{2b}.
\]
Since the determinant is nonzero, $A_n/B_n$ and $A_{n+1}/B_{n+1}$ are distinct.
Thus $a/b$ differs from at least one of them.
Denote its index by $j$.
By integrality,
\[
\left\lvert\xi-\frac ab\right\rvert
\geq
\left\lvert\frac{A_j}{B_j}-\frac ab\right\rvert
-\left\lvert\xi-\frac{A_j}{B_j}\right\rvert
\geq\frac1{2b\lvert B_j\rvert}.
\]
By the minimality of $n$, we have
$X^{(\lambda-\delta)(n-1)^2}<2b$
and hence
\[
n^2
\leq
\frac{\log_X(2b)}{\lambda-\delta}
+O(\sqrt{\log b}).
\]
Since $j\leq n+1$, we also have
$\lvert B_j\rvert\leq X^{(\kappa+\delta)(n+1)^2}$.
These two estimates give
\[
\frac{\log(2b\lvert B_j\rvert)}{\log b}
\leq
1+\frac{\kappa+\delta}{\lambda-\delta}+o(1).
\]
By our choice of $\delta$, for every sufficiently large $b$ and every $a\in\mathbb Z$ we obtain
\[
\left\lvert\xi-\frac ab\right\rvert
>b^{-1-\kappa/\lambda-\eta}.
\]
Since $\eta>0$ is arbitrary, the assertion follows.
\end{proof}

\section{Ramanujan's theta function}\label{sec:psi}

Define Ramanujan's theta function by
\[
\psi(r)
:=\sum_{m=0}^{\infty}r^{m(m+1)/2}
\qquad (|r|<1).
\]
For an integer $x$ with $|x|\geq2$, specialize the three-parameter $q$-WZ
pair of this paper by setting
$(q,\alpha,\beta,\gamma)=(x^2,x,x^2,x)$.
The target series is
\[
S_x:=\mathcal{S}_{x,x^2,x}(x^2)=\frac{P_{x,x}(x^2)-1}{1-x^{-1}}.
\]
By Gauss's identity
\[
\psi(r)=\frac{(r^2;r^2)_\infty}{(r;r^2)_\infty},
\]
the infinite product on the right-hand side is
\[
P_{x,x}(x^2)
=\frac{(x^{-2};x^{-2})_\infty}{(x^{-1};x^{-2})_\infty}
=\psi(x^{-1}).
\]
Since $1-x^{-1}$ is a nonzero rational number, $S_x$ and $\psi(x^{-1})$ are related by an affine transformation with nonzero rational coefficients.
Thus, one is irrational if and only if the other is, and, when they are irrational, they have the same irrationality measure.

\begin{theorem}\label{thm:psi}
For every integer $x$ with $|x|\geq2$, the number $\psi(x^{-1})$ is irrational, and
\[
\mu\bigl(\psi(x^{-1})\bigr)\leq\frac{18}{7}=2.571428\dots.
\]
\end{theorem}

In what follows, we prove \cref{thm:psi} using the rational approximations to $S_x$, denoted by $a_n/b_n$, obtained from our $q$-WZ pair.

\subsection{Bundschuh's estimate}

Bundschuh \cite{bundschuh1974tschakaloff} considered the irrationality of the Tschakaloff series
\[
T_x(z):=\sum_{n=0}^\infty z^nx^{-n(n-1)/2}.
\]
For an integer $x$ with $|x|\geq 2$ and a nonzero rational number $z$, his Theorem 2 (Satz 2) gives
\[
\mu\bigl(T_x(z)\bigr)\leq\frac{3+\sqrt5}{2}.
\]
Taking $z=x^{-1}$, we obtain
\[
T_x(x^{-1})
=\sum_{n=0}^\infty x^{-n(n+1)/2}
=\psi(x^{-1}),
\]
and hence Bundschuh's result yields
\[
\mu\bigl(\psi(x^{-1})\bigr)
\leq\frac{3+\sqrt5}{2}
=2.618033\dots.
\]
Zudilin~\cite{zudilin2007tschakaloff} reproved the same quantitative estimate for the Tschakaloff series by means of a hypergeometric construction of rational approximations.
Moreover, a result of Nesterenko~\cite{nesterenko1996modular} states that
$T_q(q^k)$ is transcendental for algebraic $q$ with $|q|>1$ and $k\in \mathbb{Z}$;
thus, the transcendence of $\psi(x^{-1})=T_x(x^{-1})$ is already known.
The novelty of this section is to improve the upper bound for the irrationality measure to $18/7$ for every integer $x$ with $|x|\geq2$ and to obtain the approximations from the three-parameter $q$-WZ pair.

\subsection{\texorpdfstring{The specialized $q$-WZ pair}{The specialized q-WZ pair}}

Let $X$ be an indeterminate and consider the specialization $(q,\alpha,\beta,\gamma)=(X^2,X,X^2,X)$.
Under this specialization, $F$, $G$, $b$, $c$, $a_n$, and $b_n$ become
\[
\begin{aligned}
&\mathsf F(n,k)
=-X\frac{(X^3;X^2)_k}{(X^2;X^2)_{n+k+1}},\\
&\mathsf G(n,k)
=\frac{X^{2n+1}}{X^{2n+1}-1}\mathsf F(n,k),\\
&\mathsf b(n,k)
=(-1)^kX^{k(k+1)}
\binom nk_{X^2}
\frac{(X^2;X^2)_{n+k}}
{(X^3;X^2)_k(X;X^2)_n},\\
&\mathsf c(n,k)
=\sum_{m=1}^n\mathsf G(m-1,0)+\sum_{m=1}^k \mathsf F(n,m-1),\\
&\mathsf a_n
=\sum_{k=0}^n \mathsf c(n,k) \mathsf b(n,k),\qquad
\mathsf b_n
=\sum_{k=0}^n \mathsf b(n,k).
\end{aligned}
\]

In this section, we write $F(n,k)$, $G(n,k)$, $b(n,k)$, $c(n,k)$, $a_n$, and $b_n$ for the values of these expressions at $X=x$, where $x$ is an integer with $|x|\geq 2$.

\subsection{Rational approximations}

Here, we show that $a_n/b_n\to S_x$.
Put $\rho=|x|$.
By the lemma on finite products in \cref{sec:tools}, the constants implicit in $O_\rho(1)$ below are independent of $n,k,m$.

We first estimate the weight.
For $0\leq k\leq n$, the degree of $\mathsf b(n,k)$ is given by
\[
\begin{aligned}
\deg_X \mathsf b(n,k)
={}&k(k+1)+2k(n-k)+(n+k)(n+k+1)\\
&-(k^2+2k)-n^2
=4nk-k^2+n,
\end{aligned}
\]
and \cref{lem:finite-products} therefore gives
\[
|b(n,k)|
=\rho^{4nk-k^2+n+O_\rho(1)}.
\]
Writing $k=n-j$, the difference between this exponent and that of the term with $k=n$ is $-2nj-j^2$.
Consequently,
\[
\sum_{j=1}^{n}|b(n,n-j)|
=|b(n,n)|O_\rho(\rho^{-2n}),
\]
and, for all sufficiently large $n$,
\[
b_n=b(n,n)\left(1+O_\rho(\rho^{-2n})\right)\neq0,
\qquad
|b_n|=\rho^{3n^2+n+O_\rho(1)}.
\]

Similarly, by considering the degree of $\mathsf G(m,k)$ in $X$, we obtain
\[
|G(m,k)|
=\rho^{-m^2-2mk-3m-k-1+O_\rho(1)}.
\]
In \cref{sec:construction}, we observed
\[
S_x-c(n,k)=\sum_{m=n}^{\infty}G(m,k).
\]
Moreover, the difference between the exponents of $|G(m+1,k)|$ and $|G(m,k)|$ is $-2m-2k-4$.
Thus, the sum of the absolute values of the terms with $m>n$ on the right-hand side is $O_\rho(\rho^{-2n-2k})$ times the absolute value of the term with $m=n$, and we obtain the bound, uniform in $n,k$,
\[
|S_x-c(n,k)|
\leq\rho^{-n^2-2nk-3n-k-1+O_\rho(1)}.
\]
Multiplying this by the estimate for the weight and writing $k=n-j$, we find
\[
|b(n,n-j)|\,|S_x-c(n,n-j)|
\leq\rho^{-3n-j^2+j+O_\rho(1)}.
\]
Since the series \(\sum_{j\geq0}\rho^{-j^2+j}\) converges, summing the preceding estimate gives
\[
|S_xb_n-a_n|
\leq\rho^{-3n+O_\rho(1)}.
\]
Together with the above estimate for \(b_n\), this implies
\[
\lim_{n\to\infty}\frac{a_n}{b_n}=S_x.
\]

\subsection{Estimate of the approximation error}

We now use the Casoratian to obtain a precise estimate for the error of the rational approximations constructed in the preceding subsection.
Write the common recurrence as
\[
u_{n+2}+y_1(n)u_{n+1}+y_0(n)u_n=0.
\]
For the present specialization,
\[
y_0(n)
=x^2
\frac{(x^{2n+4}-1)(x^{2n+2}-1)^2}
{(x^{2n+5}-1)(x^{2n+3}-1)^2}\cdot
\frac{x^{4n+8}+x^{2n+4}-x-1}
{x^{4n+4}+x^{2n+2}-x-1}.
\]
Since $|x|\geq 2$, we have $x^j-1\neq0$, and both the numerator and the denominator of the last fraction are positive.
Thus, $y_0(n)\neq0$.
Dividing each factor by its highest power and using \cref{lem:finite-products}, we obtain
\[
|y_0(n)|
=\rho^3\left(1+O_\rho(\rho^{-2n})\right).
\]
Set the Casoratian to be $W_n=a_nb_{n+1}-a_{n+1}b_n$.
By \cref{lem:casoratian},
$W_{n+1}=y_0(n)W_n$.
The initial value is
\[
W_0=-a_1
=\frac{x^2(x+1)(x^3+x^2+2x+1)}
{(x-1)(x^2+x+1)} \neq 0.
\]
In particular, for every $n$, we have
$W_n\neq0$ and $|W_n|=\rho^{3n+O_\rho(1)}$.
Therefore, for all sufficiently large $n$,
\[
0<\left|
\frac{a_n}{b_n}-\frac{a_{n+1}}{b_{n+1}}
\right|
=\rho^{-6n^2-5n+O_\rho(1)}.
\]
Put $d_m:=a_m/b_m-a_{m+1}/b_{m+1}$.
Since we have already shown that $a_n/b_n\to S_x$,
\[
\frac{a_n}{b_n}-S_x
=\sum_{m=n}^{\infty}d_m.
\]
The preceding estimate shows that there exists a constant $C_\rho>0$ such that, for all sufficiently large $m$,
$|d_{m+1}/d_m|\leq C_\rho\rho^{-12m}$.
Hence, for all sufficiently large $n$,
\[
0<\left|
\frac{a_n}{b_n}-S_x
\right|
=\rho^{-6n^2-5n+O_\rho(1)}.
\]
Combining this with the estimate for $b_n$ yields
\[
0<|b_nS_x-a_n|
=\rho^{-3n^2-4n+O_\rho(1)}.
\]

\subsection{Parameter symmetry}

Consider the parameter transformation at $\beta=q$ mentioned in Section~2,
\[
(\alpha,\gamma)
\longmapsto
(\alpha^*,\gamma^*)
=\left(\frac{q^2}{\alpha},\frac{\gamma q}{\alpha}\right).
\]
For the present specialization, this becomes
\[
(\alpha,\gamma)=(X,X)
\longmapsto
(\alpha^*,\gamma^*)=(X^3,X^2).
\]
We denote the objects associated with the parameters $(\alpha^*,q,\gamma^*)$ by a superscript star.
For example,
\[
\mathsf b^*(n,k)
=(-1)^kX^{k^2}
\binom nk_{X^2}
\frac{(X^2;X^2)_{n+k}}
{(X;X^2)_k(X^3;X^2)_n}.
\]
By (2.1),
\[
\mathsf b_n+(1-X^{-1})\mathsf a_n
=\frac{(X^2;X^2)_n}{(X;X^2)_n}
\sum_{k=0}^{n}\mathsf b^*(n,k).
\]
\subsection{A denominator-clearing factor}

Define the following rational function:
\[
\mathsf I_n
:=\frac{(X;X^2)_n(X;X^2)_{n+1}}
{(X^2;X^2)_n}.
\]
Using the identity
$(X^3;X^2)_k=(X;X^2)_{k+1}/(1-X)$,
we obtain
\[
\mathsf I_n\mathsf b(n,k)
=(1-X)(-1)^kX^{k(k+1)}\mathsf U_{n,k},
\tag{4.1}
\]
\[
\mathsf I_n\frac{(X^2;X^2)_n}{(X;X^2)_n}\mathsf b^*(n,k)
=(1-X)(-1)^kX^{k^2}\mathsf V_{n,k},
\tag{4.2}
\]
where
\[
\mathsf U_{n,k}
:=\frac{(X;X^2)_{n+1}(X^2;X^2)_{n+k}}
{(X;X^2)_{k+1}(X^2;X^2)_k(X^2;X^2)_{n-k}},
\]
\[
\mathsf V_{n,k}
:=\frac{(X^2;X^2)_n(X^2;X^2)_{n+k}}
{(X;X^2)_k(X^2;X^2)_k(X^2;X^2)_{n-k}}.
\]
\subsection{Common cyclotomic factors}

Define the polynomial
\[
\mathsf C_n
:=\prod_{\substack{1\leq d\leq n\\d\ \mathrm{odd}}}
\Phi_d(X)^{\lfloor n/d\rfloor}.
\]
\begin{lemma}\label{lem:psi-cyclotomic}
For $0\leq k\leq n$, we have
$\mathsf U_{n,k},\mathsf V_{n,k}\in \mathsf C_n\mathbb Z[X]$.
\end{lemma}

\begin{proof}
First,
\[
\frac{(X^2;X^2)_{n+k}}
{(X^2;X^2)_k(X^2;X^2)_{n-k}}
=\binom{n+k}{k}_{X^2}
\frac{(X^2;X^2)_n}{(X^2;X^2)_{n-k}}
\in\mathbb Z[X],
\]
and since \((X;X^2)_{k+1}\mid(X;X^2)_{n+1}\), we have \(\mathsf U_{n,k}\in\mathbb Z[X]\).
Moreover, the only cyclotomic polynomials that can occur in the denominator of \(\mathsf V_{n,k}\) are those of odd index arising from \((X;X^2)_k\).

Let \(\nu_d\) denote the order of \(\Phi_d(X)\) in a rational function.
For odd \(d\), counting the factors in the finite products gives
\[
\nu_d\bigl((X^2;X^2)_j\bigr)=\left\lfloor\frac jd\right\rfloor,
\qquad
\nu_d\bigl((X;X^2)_j\bigr)=\left\lfloor
\frac{j+(d-1)/2}{d}
\right\rfloor.
\]
Let \(d=2s+1\), and let \(r,t\) be the least nonnegative residues of \(n,k\), respectively, modulo \(d\).
Substituting the preceding formulas into \(\mathsf U_{n,k},\mathsf V_{n,k}\), we obtain
\[
\begin{aligned}
\nu_d(\mathsf U_{n,k})-\left\lfloor\frac nd\right\rfloor
={}&\mathbf1_{r+t\geq d}+\mathbf1_{r<t}
+\mathbf1_{r\geq s}-\mathbf1_{t\geq s},\\
\nu_d(\mathsf V_{n,k})-\left\lfloor\frac nd\right\rfloor
={}&\mathbf1_{r+t\geq d}+\mathbf1_{r<t}
-\mathbf1_{t\geq s+1}.
\end{aligned}
\]
Here, \(\mathbf1_E\) denotes the indicator of the condition \(E\).
If the negative term occurs in the first formula, then \(t\geq s\).
In this case, if \(r\geq s\), the third term cancels the negative term, while if \(r<s\), then \(r<t\) does so.
If the negative term occurs in the second formula, then \(t\geq s+1\).
In this case, unless \(r<t\), we have \(r\geq t\), and hence \(r+t\geq2s+2>d\).
Thus, the right-hand sides of both formulas are nonnegative, and we obtain $\nu_d(\mathsf U_{n,k}),\nu_d(\mathsf V_{n,k})\geq\lfloor n/d\rfloor$.
The case $d=1$ is also included in the same formulas by taking $s=r=t=0$.
Taking the product over all odd $d$ proves the asserted divisibility by $\mathsf C_n$.
\end{proof}

\subsection{Clearing denominators in the finite sums}

Summing (4.1) and (4.2) over $k$ and using $\mathsf C_n\mid \mathsf U_{n,k},\mathsf V_{n,k}$, we obtain
\[
\mathsf I_n\mathsf b_n,\ 
\mathsf I_n\bigl(\mathsf b_n+(1-X^{-1})\mathsf a_n\bigr)
\in(1-X)\mathsf C_n\mathbb Z[X].
\]
Taking the difference of these two polynomials gives
\[
(1-X^{-1})\mathsf I_n\mathsf a_n\in(1-X)\mathsf C_n\mathbb Z[X].
\]
Since $1-X^{-1}=-(1-X)/X$, dividing both sides by $1-X$ yields
$\mathsf I_n\mathsf a_n\in \mathsf C_n\mathbb Z[X]$.
Although $\mathsf J_n:=\mathsf I_n/\mathsf C_n$ is in general a rational function, the above divisibility implies
\[
\mathsf J_n\mathsf a_n\in\mathbb Z[X],
\qquad
\mathsf J_n\mathsf b_n\in\mathbb Z[X].
\]

\subsection{The size of the denominator}

We write $I_n$, $J_n$, $C_n$ for the values of $\mathsf I_n$, $\mathsf J_n$, $\mathsf C_n$ at $X=x$.
The degree of $\mathsf I_n$ in $X$ is
\[
n^2+(n+1)^2-n(n+1)
=n^2+n+1.
\]
Thus, by \cref{lem:finite-products}, for a fixed integer $x$ with $|x|\geq2$,
\[
|I_n|
=\rho^{n^2+n+1+O_\rho(1)}.
\]
On the other hand,
\[
\deg \mathsf C_n
=\sum_{\substack{d\leq n\\d\ \mathrm{odd}}}
\left\lfloor\frac nd\right\rfloor\varphi(d)
=\frac13n^2+O(n).
\]
Indeed, rearranging the first sum in divisor order shows that the contribution of each integer $m\leq n$ is the largest odd divisor of $m$.
Writing $m=2^eu$ uniquely with $u$ odd, we obtain
\[
\sum_{e\geq0}
\sum_{\substack{u\leq n/2^e\\u\ \mathrm{odd}}}u
=
\sum_{e\geq0}
\left(
\frac{n^2}{4\cdot4^e}
+O\left(\frac{n}{2^e}\right)
\right)
=\frac13n^2+O(n).
\]

Set $m_{d,n}=\lfloor n/d\rfloor$ if $d$ is odd and $m_{d,n}=0$ if it is even.
Applying \cref{lem:cyclotomic-growth} gives
$|C_n|=\rho^{n^2/3+O_\rho(n\log n)}$ and hence
\[
|J_n|
=\rho^{(2/3)n^2+o(n^2)}.
\]

\subsection{Irrationality and irrationality measure}

For an integer $x$ with $|x|\geq2$, set
\[
A_n:=J_na_n,
\qquad
B_n:=J_nb_n.
\]
By the polynomiality established above, $A_n,B_n\in\mathbb Z$.
We have $a_n/b_n\to S_x$ and
\[
|b_n|=\rho^{3n^2+n+O_\rho(1)}.
\]
Combining these facts with the estimate for $|J_n|$, we obtain
\[
|A_n|+|B_n|
=\rho^{(11/3)n^2+o(n^2)},
\]
\[
0<|B_nS_x-A_n|
=\rho^{-(7/3)n^2+o(n^2)}.
\]
Since $J_n\neq0$, $b_n\neq0$ for all sufficiently large $n$, and $W_n\neq0$, we have $B_n\neq0$ and
\[
A_nB_{n+1}-A_{n+1}B_n
=J_nJ_{n+1}W_n
\neq0
\]
for all sufficiently large $n$.
Applying \cref{lem:irrationality-measure} with base $\rho$ and
$\kappa=11/3$,
$\lambda=7/3$,
we conclude that $S_x$ is irrational and that
\[
\mu\bigl(\psi(x^{-1})\bigr)=\mu(S_x)
\leq1+\frac{11/3}{7/3}
=\frac{18}{7}.
\]
This proves \cref{thm:psi}.

\section{The generating function of the divisor function restricted to odd integers}\label{sec:delta}

For a positive integer $n$, let $d(n)$ denote the number of its positive divisors, and write the generating function of the divisor function restricted to odd integers as
\[
\Delta(r)=\sum_{m=0}^\infty d(2m+1)r^m=\sum_{k=0}^\infty\frac{r^k}{1-r^{2k+1}},
\]
where $|r|<1$.
Indeed, expanding the right-hand side as an absolutely convergent double series gives
\[
\sum_{k=0}^\infty\frac{r^k}{1-r^{2k+1}}
=\sum_{j,k\geq0}r^{k+(2k+1)j}.
\]
The equality $m=k+(2k+1)j$ is equivalent to
\[
2m+1=(2j+1)(2k+1),
\]
so the coefficient of $r^m$ equals the number of positive divisors of $2m+1$.
For an integer $x$ with $|x|\geq2$, specialize the three-parameter $q$-WZ
pair of this paper by setting
$(q,\alpha,\beta,\gamma)=(x^2,x^2,x,x)$.
The target series is
\[
S_x:=\mathcal{S}_{x^2,x,x}(x^2)
=\ell_{x^2}(x,x)
=\sum_{k=0}^\infty \frac{x^{k+1}}{x^{2k+1}-1}
=\Delta(x^{-1}).
\]

\begin{theorem}\label{thm:delta}
For every integer $x$ with $|x|\geq2$, the number $\Delta(x^{-1})$ is irrational, and
\[
\mu\bigl(\Delta(x^{-1})\bigr)\leq
\frac{18\pi^2}{7\pi^2-24}
=3.9402038235\dots.
\]
\end{theorem}

In what follows, we prove \cref{thm:delta} using the rational approximations to $S_x$, denoted by $a_n/b_n$, obtained from our $q$-WZ pair.

\subsection{The Bundschuh--Zudilin estimate}

Theorem 4 of Bundschuh--Zudilin~\cite{bundschuhzudilin2008rational} gives
\[
\mu\bigl(\ell_p(u,u)\bigr)\leq6
\]
for an integer $p\notin\{0,\pm1\}$ and a rational number $u$ satisfying $0<|u|<|p|$.
Taking $(p,u)=(x^2,x)$, the hypotheses are satisfied for every integer $x$ with $|x|\geq2$, and we obtain
\[
\mu(\Delta\bigl(x^{-1})\bigr)=\mu\bigl(\ell_{x^2}(x,x)\bigr)\leq6.
\]
The bound in \cref{thm:delta} improves this estimate.

Theorem 1 of Zudilin~\cite{zudilin2016generalized} shows that $\ell_p(u,v)$ is irrational for an integer $p$ with $|p|>1$ and nonzero rational numbers $u,v$, provided that $u\notin\{p,p^2,\ldots\}$ and $|v|<|p|$.
The choice $(p,u,v)=(x^2,x,x)$ satisfies these conditions for every integer $x$ with $|x|\geq2$, so the irrationality of $\Delta(x^{-1})$ also follows from this general theorem.

\subsection{\texorpdfstring{The specialized $q$-WZ pair}{The specialized q-WZ pair}}

Let $X$ be an indeterminate and consider the specialization $(q,\alpha,\beta,\gamma)=(X^2,X^2,X,X)$.
Under this specialization, $F$, $G$, $b$, $c$, $a_n$, and $b_n$ become
\[
\begin{aligned}
&\mathsf F(n,k)
=-X^{k+1}
\frac{(X;X^2)_k(X^2;X^2)_n}
{(X;X^2)_{n+k+1}(X;X^2)_n},\\
&\mathsf G(n,k)
=\frac{X^{2n+1}}{X^{2n+1}-1}\mathsf F(n,k),\\
&\mathsf b(n,k)
=(-1)^kX^{k^2}
\frac{(X;X^2)_{n+k}}
{(X;X^2)_k(X^2;X^2)_k(X^2;X^2)_{n-k}},\\
&\mathsf c(n,k)
=\sum_{m=1}^n\mathsf G(m-1,0)+\sum_{m=1}^k \mathsf F(n,m-1),\\
&\mathsf a_n
=\sum_{k=0}^n \mathsf c(n,k) \mathsf b(n,k),\qquad
\mathsf b_n
=\sum_{k=0}^n \mathsf b(n,k).
\end{aligned}
\]

In this section, we write $F(n,k)$, $G(n,k)$, $b(n,k)$, $c(n,k)$, $a_n$, and $b_n$ for the values of these expressions at $X=x$, where $x$ is an integer with $|x|\geq 2$.

\subsection{Rational approximations}

Here, we show that $a_n/b_n\to S_x$.
Put $\rho=|x|$.
By the lemma on finite products in \cref{sec:tools}, the constants implicit in $O_\rho(1)$ below are independent of $n,k,m$.

We first estimate the weight.
Since $\deg_X(X;X^2)_N=N^2$ and $\deg_X(X^2;X^2)_N=N(N+1)$, we have
\[
\begin{aligned}
\deg_X \mathsf b(n,k)
={}&k^2+(n+k)^2-k^2-k(k+1)-(n-k)(n-k+1)\\
={}&4nk-k^2-n.
\end{aligned}
\]
Applying \cref{lem:finite-products}, we obtain
\[
|b(n,k)|
=\rho^{4nk-k^2-n+O_\rho(1)}.
\]
Writing $k=n-j$, the difference between this exponent and that of the term with $k=n$ is $-2nj-j^2$.
Hence,
\[
\sum_{j=1}^{n}|b(n,n-j)|
=|b(n,n)|O_\rho(\rho^{-2n}),
\]
and, for all sufficiently large $n$,
\[
b_n=b(n,n)\left(1+O_\rho(\rho^{-2n})\right)\neq0,
\qquad
|b_n|=\rho^{3n^2-n+O_\rho(1)}.
\]

Similarly, by considering the degree of $\mathsf G(m,k)$ in $X$, we obtain
\[
|G(m,k)|
=\rho^{-m^2-2mk-m-k+O_\rho(1)}.
\]
In \cref{sec:construction}, we observed
\[
S_x-c(n,k)=\sum_{m=n}^{\infty}G(m,k).
\]
Moreover, the difference between the exponents of $G(m+1,k)$ and $G(m,k)$ is $-2m-2k-2$.
Thus, the sum of the absolute values of the terms with $m>n$ on the right-hand side is $O_\rho(\rho^{-2n-2k})$ times the absolute value of the term with $m=n$, and we obtain the bound, uniform in $n,k$,
\[
|S_x-c(n,k)|
\leq\rho^{-n^2-2nk-n-k+O_\rho(1)}.
\]
Multiplying this by the estimate for the weight and writing $k=n-j$, we find
\[
|b(n,n-j)|\,|S_x-c(n,n-j)|
\leq\rho^{-3n-j^2+j+O_\rho(1)}.
\]
Since the series $\sum_{j\geq0}\rho^{-j^2+j}$ converges, summing the preceding estimate gives
\[
|S_xb_n-a_n|
\leq\rho^{-3n+O_\rho(1)}.
\]
Together with the above estimate for $b_n$, this implies
\[
\lim_{n\to\infty}\frac{a_n}{b_n}=S_x.
\]

\subsection{Estimate of the approximation error}

We now use the Casoratian to obtain a precise estimate for the error of the rational approximations constructed in the preceding subsection.
Write the common recurrence as
\[
u_{n+2}+y_1(n)u_{n+1}+y_0(n)u_n=0.
\]
Specializing the coefficient $y_0(n)$ from Section~2 to the present parameters gives
\[
y_0(n)
=x^2
\frac{(x^{2n+1}-1)(x^{2n+2}-1)}
{(x^{2n+3}-1)(x^{2n+4}-1)}\cdot
\frac{x^{4n+7}+x^{2n+4}-2}
{x^{4n+3}+x^{2n+2}-2}.
\]
One easily checks that $y_0(n)>0$ for every $n$.
Moreover,
\[
y_0(n)=\rho^2\left(1+O_\rho(\rho^{-2n})\right).
\]
Set the Casoratian to be $W_n:=a_nb_{n+1}-a_{n+1}b_n$.
By the lemma on Casoratians in Section~3,
$W_{n+1}=y_0(n)W_n$.
The initial value is
\[
W_0=\frac{x^2(x^2+2x+2)}{x^2-1}>0.
\]
The above sign check shows that $W_n>0$ for every $n$.
Canceling factors in the product of the $y_0(n)$ gives
\[
W_n
=\frac{x^{2n+2}
\left(x^{4n+3}+x^{2n+2}-2\right)}
{(x^{2n+1}-1)(x^{2n+2}-1)}.
\]
In particular,
$|W_n|=\rho^{2n+2+O_\rho(1)}$.
Therefore, for all sufficiently large $n$,
\[
0<
\left|
\frac{a_n}{b_n}-\frac{a_{n+1}}{b_{n+1}}
\right|
=\frac{|W_n|}{|b_nb_{n+1}|}
=\rho^{-6n^2-2n+O_\rho(1)}.
\]
Put $d_r:=a_r/b_r-a_{r+1}/b_{r+1}$.
Since we have already shown that $a_n/b_n\to S_x$,
\[
\frac{a_n}{b_n}-S_x
=\sum_{r=n}^{\infty}d_r.
\]
The preceding estimate shows that there exists a constant $C_\rho>0$ such that, for all sufficiently large $r$,
$|d_{r+1}/d_r|\leq C_\rho\rho^{-12r}$.
Hence, for all sufficiently large $n$,
\[
\sum_{r=n+1}^{\infty}|d_r|<\frac12|d_n|.
\]
This inequality gives both the nonvanishing of the sum and upper and lower bounds for it, and therefore
\[
0<
\left|
\frac{a_n}{b_n}-S_x
\right|
=\rho^{-6n^2-2n+O_\rho(1)}.
\]
Combining this with the estimate for $b_n$ yields
\[
0<|b_nS_x-a_n|
=\rho^{-3n^2-3n+O_\rho(1)}.
\]

\subsection{A common denominator}

We define the following polynomials:
\[
\mathsf U_n
:=\prod_{\substack{1\leq d\leq2n-1\\d\ \mathrm{odd}}}
\Phi_d(X),\qquad
\mathsf E_n
:=\prod_{e=1}^{n}
\Phi_{2e}(X)^{\lfloor n/e\rfloor},
\]
and
$\mathsf D_n:=\mathsf U_n\mathsf E_n$.

\begin{proposition}\label{prop:delta-integrality}
For every $n\geq0$, we have
$\mathsf D_n\mathsf a_n,\mathsf D_n\mathsf b_n\in\mathbb Z[X]$.
\end{proposition}

We prove \cref{prop:delta-integrality} below.

\subsection{The denominator of the weight}

Let $\nu_d$ denote the order of $\Phi_d(X)$ in a rational function.
For odd $d$,
\[
\nu_d\bigl((X^2;X^2)_N\bigr)
=\left\lfloor\frac Nd\right\rfloor,
\qquad
\nu_d\bigl((X;X^2)_N\bigr)
=\left\lfloor
\frac{N+(d-1)/2}{d}
\right\rfloor.
\]
Therefore,
\[
\begin{aligned}
\nu_d(\mathsf b(n,k))
={}&\left\lfloor\frac{n+k+(d-1)/2}{d}\right\rfloor
-\left\lfloor\frac{k+(d-1)/2}{d}\right\rfloor
-\left\lfloor\frac kd\right\rfloor
-\left\lfloor\frac{n-k}{d}\right\rfloor\\
\geq{}&\left\lfloor\frac nd\right\rfloor
-\left\lfloor\frac kd\right\rfloor
-\left\lfloor\frac{n-k}{d}\right\rfloor
\geq0.
\end{aligned}
\]
The first inequality follows because the difference of the first two terms counts the number of times a given residue class occurs in an interval of $n$ consecutive integers.
Thus, no cyclotomic polynomial of odd index remains in the denominator of $b(n,k)$.

Now let the even index be $d=2e$.
The product $(X;X^2)_N$ has no cyclotomic factor of even index, and
\[
\nu_{2e}\bigl((X^2;X^2)_N\bigr)
=\left\lfloor\frac Ne\right\rfloor.
\]
Hence,
\[
\nu_{2e}(b(n,k))
=-\left\lfloor\frac ke\right\rfloor
-\left\lfloor\frac{n-k}{e}\right\rfloor
\geq-\left\lfloor\frac ne\right\rfloor.
\]
It follows that, for every $0\leq k\leq n$,
\[
\mathsf E_n\mathsf b(n,k)\in\mathbb Z[X].
\]
In particular, summing over $k$ gives $\mathsf E_n\mathsf b_n\in\mathbb Z[X]$.

\subsection{\texorpdfstring{The part of $\mathsf a_n$ arising from $\mathsf F$}{The part of a n arising from F}}

Adding the two difference relations for the potential along the path from $(0,0)$ to $(0,k)$ and then to $(n,k)$ gives
\[
\mathsf c(n,k)
=\sum_{j=1}^{k}\frac{X^j}{X^{2j-1}-1}
+\sum_{m=1}^{n}\mathsf G(m-1,k).
\]
The two difference relations ensure that summing along either path from $(0,0)$ to $(n,k)$ gives the same value, so this sum equals the defined potential $\mathsf c(n,k)$.
Thus, $\mathsf a_n$ can be decomposed as
\[
\begin{aligned}
\mathsf a_n&=\mathsf a_n^{(\mathsf F)}+\mathsf a_n^{(\mathsf G)},\\
\mathsf a_n^{(\mathsf F)}
&:=\sum_{k=0}^{n}\mathsf b(n,k)
\sum_{j=1}^{k}\frac{X^j}{X^{2j-1}-1},\\
\mathsf a_n^{(\mathsf G)}
&:=\sum_{m=1}^{n}\sum_{k=0}^{n}\mathsf b(n,k)\mathsf G(m-1,k).
\end{aligned}
\]
The least common multiple of $X^{2j-1}-1\ (1\leq j\leq n)$ is $\mathsf U_n$.
Therefore,
\[
\mathsf U_n\sum_{j=1}^{k}\frac{X^j}{X^{2j-1}-1}
\in\mathbb Z[X]
\qquad(0\leq k\leq n).
\]
Together with the denominator of the weight found above, this gives
\[
\mathsf U_n\mathsf E_n\mathsf a_n^{(\mathsf F)}\in\mathbb Z[X].
\]

\subsection{\texorpdfstring{The part of $\mathsf a_n$ arising from $\mathsf G$}{The part of a n arising from G}}

For $1\leq m\leq n$, set
\[
\mathsf T_{n,m}
:=\sum_{k=0}^{n}\mathsf b(n,k)\mathsf G(m-1,k).
\]
Substituting the definitions of $\mathsf b(n,k)$ and $\mathsf G(m-1,k)$, we obtain
\[
\mathsf T_{n,m}
=X^{2m}\frac{(X^2;X^2)_{m-1}}{(X;X^2)_m}
\cdot
\sum_{k=0}^{n}
(-1)^kX^{k^2+k}
\frac{(X^{2m+2k+1};X^2)_{n-m}}
{(X^2;X^2)_k(X^2;X^2)_{n-k}}.
\]
We use the finite $q$-binomial theorem in the form
\[
(z;Q)_N
=\sum_{j=0}^{N}
(-1)^jQ^{j(j-1)/2}\binom Nj_Qz^j.
\]
Taking $Q=X^2$ and $z=X^{2m+2k+1}$ gives
\[
(X^{2m+2k+1};X^2)_{n-m}
=\sum_{j=0}^{n-m}
(-1)^jX^{j^2+2mj+2kj}
\binom{n-m}{j}_{X^2}.
\]
Substituting this, we obtain
\[
\begin{aligned}
\mathsf T_{n,m}
={}&X^{2m}\frac{(X^2;X^2)_{m-1}}{(X;X^2)_m}
\sum_{j=0}^{n-m}
(-1)^jX^{j^2+2mj}
\binom{n-m}{j}_{X^2}\\
&\times\frac1{(X^2;X^2)_n}
\sum_{k=0}^{n}
(-1)^kX^{k^2+k+2kj}
\binom nk_{X^2}.
\end{aligned}
\]
By the $q$-binomial theorem, the sum over $k$ equals $(X^{2j+2};X^2)_n$.
Using further that
\[
\frac{(X^{2j+2};X^2)_n}{(X^2;X^2)_n}
=\binom{n+j}{j}_{X^2},
\]
we obtain
\[
\mathsf T_{n,m}
=X^{2m}
\frac{(X^2;X^2)_{m-1}}{(X;X^2)_m}
\sum_{j=0}^{n-m}
(-1)^jX^{j^2+2mj}
\binom{n-m}{j}_{X^2}
\binom{n+j}{j}_{X^2}.
\]
The final finite sum belongs to $\mathbb Z[X]$.
For odd $d$,
\[
\nu_d\biggl(
\frac{(X^2;X^2)_{m-1}}{(X;X^2)_m}
\biggr)
=\biggl\lfloor\frac{m-1}{d}\biggr\rfloor
-\biggl\lfloor\frac{m+(d-1)/2}{d}\biggr\rfloor
\geq-1.
\]
If this order is negative, then $\Phi_d(X)$ is a factor of $(X;X^2)_m$, so $d\leq2m-1\leq2n-1$.
Therefore,
$\mathsf U_n\mathsf T_{n,m}\in\mathbb Z[X]$,
and summing over $m$ gives
\[
\mathsf U_n\mathsf a_n^{(\mathsf G)}\in\mathbb Z[X].
\]
Combining the two parts, we obtain $\mathsf D_n\mathsf a_n\in\mathbb Z[X]$.
This proves the proposition.

\subsection{The size of the denominator}

We write $U_n$, $E_n$, $D_n$ for the values of $\mathsf U_n$, $\mathsf E_n$, $\mathsf D_n$ at $X=x$.

\begin{lemma}
\[
\sum_{\substack{d\leq y\\d\ \mathrm{odd}}}\varphi(d)
=
\frac{2}{\pi^2}y^2+O(y\log y).
\]
\end{lemma}

\begin{proof}
Let $\mu$ denote the M\"obius function.
Using $\varphi(d)=d\sum_{a\mid d}\mu(a)/a$ and
$\sum_{b\leq y\text{ odd}}b
=y^2/4+O(y)$,
we obtain
\[
\sum_{\substack{d\leq y\\d\ \mathrm{odd}}}\varphi(d)
=\sum_{\substack{a\leq y\\a\ \mathrm{odd}}}
\mu(a)
\sum_{\substack{b\leq y/a\\b\ \mathrm{odd}}}b
=\frac{y^2}{4}
\sum_{\substack{a\geq1\\a\ \mathrm{odd}}}\frac{\mu(a)}{a^2}
+O(y\log y).
\]
Combining this with
\[
\sum_{\substack{a\geq1\\a\ \mathrm{odd}}}\frac{\mu(a)}{a^2}
=\prod_{p\ \mathrm{odd}}(1-p^{-2})
=\frac8{\pi^2},
\]
we obtain the desired estimate.
\end{proof}

Taking $y=2n-1$ gives
\[
\deg \mathsf U_n
=\sum_{\substack{1\leq d\leq2n-1\\d\ \mathrm{odd}}}
\varphi(d)
=\frac8{\pi^2}n^2+O(n\log n).
\]
On the other hand, the degree of $\mathsf E_n$ is
\[
\deg \mathsf E_n
=\sum_{e=1}^{n}\left\lfloor\frac ne\right\rfloor\varphi(2e)
=\frac23n^2+O(n).
\]
Indeed, write
\[
\sum_{e=1}^{n}\left\lfloor\frac ne\right\rfloor\varphi(2e)
=\sum_{m=1}^{n}\sum_{e\mid m}\varphi(2e),
\]
and let $m=2^as$, where $s$ is odd.
Then
\[
\sum_{e\mid m}\varphi(2e)
=(2^{a+1}-1)s
=2m-s.
\]
Here, $s$ is the largest odd divisor of $m$.
Denoting this largest odd divisor by $s(m)$, we have
\[
\begin{aligned}
\sum_{m=1}^{n}s(m)
&=\sum_{a\geq0}
\sum_{\substack{s\leq n/2^a\\s\ \mathrm{odd}}}s\\
&=\sum_{a\geq0}
\left(\frac{n^2}{4\cdot4^a}
+O\left(\frac{n}{2^a}\right)\right)
=\frac{n^2}{3}+O(n).
\end{aligned}
\]
Substituting this into the above divisor sum gives $\deg \mathsf E_n=2n^2/3+O(n)$.
It follows that
\[
\deg \mathsf D_n
=\left(\frac23+\frac8{\pi^2}\right)n^2
+O(n\log n).
\]
When $\mathsf D_n$ is written as a product of cyclotomic polynomials, the exponent of each $\Phi_d$ is at most $2n/d$ for $d\leq2n$.
Therefore, applying the lemma on products of cyclotomic polynomials in Section~3 with $N=2n$ gives, for an integer $x$ with $|x|\geq2$,
\[
|D_n|
=\rho^{(2/3+8/\pi^2)n^2+O_\rho(n\log n)}.
\]
Moreover, the zeros of cyclotomic polynomials have absolute value $1$, so $D_n\neq0$.

\subsection{Irrationality and irrationality measure}

For an integer $x$ with $|x|\geq2$, set
\[
A_n:=D_na_n,
\qquad
B_n:=D_nb_n.
\]
By the polynomiality established above, $A_n,B_n\in\mathbb Z$.
Combining $a_n/b_n\to S_x$ with the estimates for $|b_n|$ and $|D_n|$, we obtain
\[
|A_n|+|B_n|
\leq\rho^{(11/3+8/\pi^2)n^2+o(n^2)},
\]
\[
0<|B_nS_x-A_n|
=\rho^{-(7/3-8/\pi^2)n^2+o(n^2)}.
\]
Here,
\[
\lambda:=\frac73-\frac8{\pi^2}>0.
\]
Since $D_n(x)\neq0$, $b_n\neq0$ for all sufficiently large $n$, and $W_n>0$, we have $B_n\neq0$ and
\[
A_nB_{n+1}-A_{n+1}B_n
=D_n(x)D_{n+1}(x)W_n\neq0
\]
for all sufficiently large $n$.
Applying \cref{lem:irrationality-measure} with base $\rho$ and
\[
\kappa=\frac{11}{3}+\frac{8}{\pi^2},
\qquad
\lambda=\frac{7}{3}-\frac{8}{\pi^2},
\]
we conclude that $S_x$ is irrational and that
\[
\mu\bigl(\Delta(x^{-1})\bigr)=\mu(S_x)
\leq
1+\frac{11/3+8/\pi^2}{7/3-8/\pi^2}
=\frac{18\pi^2}{7\pi^2-24}.
\]
This proves \cref{thm:delta}.

\section{The generating function for \texorpdfstring{$4$-regular partitions}{4-regular partitions}}\label{sec:b4}

For a nonnegative integer $m$, let $b_4(m)$ denote the number of partitions of $m$ in which no part is divisible by $4$.
Set $b_4(0)=1$, corresponding to the empty partition, and write its generating function as
\[
B_4(r)
:=\sum_{m=0}^{\infty}b_4(m)r^m
=\frac{(r^4;r^4)_\infty}{(r;r)_\infty}
\qquad (|r|<1).
\]
For an integer $x$ with $|x|\geq 2$, specialize the three-parameter $q$-WZ pair of this paper by setting
$(q,\alpha,\beta,\gamma)=(x^2,-x,x^2,x)$.
The target series then becomes
\[
S_x:=\mathcal{S}_{-x,x^2,x}(x^2)=\frac{P_{-x,x}(x^2)-1}{1+x^{-1}}.
\]
The infinite product $P_{-x,x}(x^2)$ on the right is
\[
P_{-x,x}(x^2)
=\frac{(-x^{-2};x^{-2})_\infty}{(x^{-1};x^{-2})_\infty}
=\frac{(x^{-4};x^{-4})_\infty}{(x^{-1};x^{-1})_\infty}
=B_4(x^{-1}).
\]
Since $1+x^{-1}$ is a nonzero rational number, $S_x$ and $B_4(x^{-1})$ are related by an affine transformation with nonzero rational coefficients.
Their irrationality is therefore equivalent, and if they are irrational, their irrationality measures are equal.

\begin{theorem}\label{thm:b4}
For every integer $x$ with $|x|\geq2$, the number $B_4(x^{-1})$ is irrational and satisfies
\[
\mu\!\left(B_4(x^{-1})\right)\leq3.
\]
\end{theorem}

In what follows, we prove \cref{thm:b4} using the rational approximations to $S_x$, denoted by $a_n/b_n$, obtained from our $q$-WZ pair.

\subsection{Rochev's estimate}

For comparison, we record the upper bound obtained from a result of Rochev~\cite{rochev2010linear}.
In the notation of this paper, the series considered by Rochev is
\[
f(z):=\sum_{n=0}^{\infty}
\frac{z^n}{\prod_{k=1}^n(x^{2k}-1)}
=(-zx^{-2};x^{-2})_\infty,
\qquad
B_4(x^{-1})=\frac{f(1)}{f(-x)}.
\]
Here we used Euler's formula
\[
(-t;Q)_\infty
=\sum_{n=0}^{\infty}
\frac{Q^{n(n-1)/2}t^n}{(Q;Q)_n}
\]
with $Q=x^{-2}$ and $t=zx^{-2}$.

In the notation of Rochev's theorem, take
\[
q=x^2,\quad P(z)=z-1,\quad d=1,\quad
m=2,\quad (\alpha_1,\alpha_2)=(1,-x),\quad
s_1=s_2=1.
\]
Then $S=s_1+s_2=2$, and since $q=x^2/1$, the parameter $\gamma$ in that theorem is $0$.
Moreover, since $P$ is not a monomial, the constant $M$ and the exponent of the linear form in that theorem are
\[
M=dS+1+\sqrt{dS(dS+1)}=3+\sqrt6,
\qquad
\frac{M-1}{1-M\gamma}=2+\sqrt6.
\]
If $|x|\geq2$, then $P(x^{2n})\neq0$ for every $n\geq1$, and the conditions $\alpha_1/\alpha_2\notin q^{\mathbb Z}$ and $\alpha_j\notin P(0)q^{\mathbb Z_{>0}}$ also hold.
Thus all the hypotheses of the theorem are satisfied.
Consequently, for every $\varepsilon>0$ and all sufficiently large $H=\max\{|A_1|,|A_2|,2\}$,
\[
0<|A_0+A_1f(1)+A_2f(-x)|
\geq H^{-(2+\sqrt6+\varepsilon)},
\]
where $A_0,A_1,A_2\in\mathbb Z$ are not all zero.

Set $A_0=0$, $A_1=b$, and $A_2=-a$, with $b>0$.
If $a/b$ lies in a fixed bounded neighborhood of $B_4(x^{-1})$, then $|a|=O_x(b)$ and hence $H=O_x(b)$.
Outside this neighborhood, $|B_4(x^{-1})-a/b|$ has a positive lower bound, so it is enough to consider points within the neighborhood when estimating the irrationality measure.
Dividing the preceding inequality by $b|f(-x)|$ gives, for some constant $C_{x,\varepsilon}>0$,
\[
0<\left|B_4(x^{-1})-\frac ab\right|
\geq C_{x,\varepsilon}b^{-(3+\sqrt6+\varepsilon)}.
\]
This yields
\[
\mu\!\left(B_4(x^{-1})\right)
\leq3+\sqrt6
=5.449489\dots.
\]
The theorem of this section improves this upper bound.

\subsection{The specialized \texorpdfstring{$q$-WZ pair}{q-WZ pair}}

Let $X$ be an indeterminate and consider the specialization $(q,\alpha,\beta,\gamma)=(X^2,-X,X^2,X)$.
Under this specialization, $F$, $G$, $b$, $c$, $a_n$, and $b_n$ become
\[
\begin{aligned}
&\mathsf F(n,k)
=(-1)^{k+1}X
\frac{(-X^3;X^2)_k(-X;X^2)_n}
{(X^2;X^2)_{n+k+1}(X;X^2)_n},\\
&\mathsf G(n,k)
=\frac{X^{2n+1}}{X^{2n+1}-1}\mathsf F(n,k),\\
&\mathsf b(n,k)
=X^{k(k+1)}
\binom nk_{X^2}
\frac{(X^2;X^2)_{n+k}}
{(-X^3;X^2)_k(-X;X^2)_n},\\
&\mathsf c(n,k)
=\sum_{m=1}^n\mathsf G(m-1,0)+\sum_{m=1}^k \mathsf F(n,m-1),\\
&\mathsf a_n
=\sum_{k=0}^n \mathsf c(n,k) \mathsf b(n,k),\qquad
\mathsf b_n
=\sum_{k=0}^n \mathsf b(n,k).
\end{aligned}
\]

In this section, we write $F(n,k)$, $G(n,k)$, $b(n,k)$, $c(n,k)$, $a_n$, and $b_n$ for the values of these expressions at $X=x$, where $x$ is an integer with $|x|\geq 2$.

\subsection{Rational approximations}

Here, we show that $a_n/b_n\to S_x$.
Put $\rho=|x|$.
By the lemma on finite products in \cref{sec:tools}, the constants implicit in $O_\rho(1)$ below are independent of $n,k,m$.

We first estimate the weight.
For $0\leq k\leq n$, comparison of the degrees in $X$ in $\mathsf b(n,k)$ gives
\[
b(n,n)>0,
\qquad
|b(n,k)|
=\rho^{4nk-k^2+n+O_\rho(1)}.
\]
Indeed, $x^{n(n+1)}>0$, and $(x^2;x^2)_{2n}$ contains an even number of negative factors.
If $x<0$, both finite products in the denominator have sign $(-1)^n$, while if $x>0$, both are positive; hence $b(n,n)>0$ in either case.
Writing $k=n-j$, the difference between the exponent of this term and that of the $k=n$ term is $-2nj-j^2$.
The sum of the absolute values of all terms except the last is $O_\rho(\rho^{-2n})$ times the last term, and therefore, for all sufficiently large $n$,
\[
b_n
=b(n,n)\left(1+O_\rho(\rho^{-2n})\right)>0,
\qquad
b_n=\rho^{3n^2+n+O_\rho(1)}.
\]

Similarly, comparison of the degrees in $X$ in $\mathsf G(m,k)$ gives
\[
\operatorname{sgn}G(m,k)
=
\begin{cases}
1,&x>0,\\
(-1)^{k+1},&x<0,
\end{cases}
\qquad
|G(m,k)|
=\rho^{-m^2-2mk-3m-k-1+O_\rho(1)}.
\]
For each fixed $k$, all terms have the same sign.
Moreover, increasing $m$ to $m+1$ changes the exponent by $-2m-2k-4$, so the sum of the absolute values of the terms with $m>n$ is $O_\rho(\rho^{-2n-2k})$ times the absolute value of the term with $m=n$.
Combining this with the representation obtained in Section~2,
\[
S_x-c(n,k)=\sum_{m=n}^{\infty}G(m,k),
\]
we obtain, uniformly in $n,k$,
\[
0<|S_x-c(n,k)|
=\rho^{-n^2-2nk-3n-k-1+O_\rho(1)}.
\]
Multiplying this by the estimate for the weight and putting $k=n-j$ gives
\[
|b(n,n-j)|\,|S_x-c(n,n-j)|
=\rho^{-3n-j^2+j+O_\rho(1)}.
\]
Since the series $\sum_{j\geq0}\rho^{-j^2+j}$ converges, summing the preceding estimate yields
\[
|S_xb_n-a_n|
\leq\rho^{-3n+O_\rho(1)}.
\]
Together with the preceding estimate for $b_n$, this proves
\[
\lim_{n\to\infty}\frac{a_n}{b_n}=S_x.
\]

\subsection{Estimate of the approximation error}

We now use the Casoratian to obtain a precise estimate for the error of the rational approximations constructed in the preceding subsection.
Write the common recurrence as
\[
u_{n+2}+y_1(n)u_{n+1}+y_0(n)u_n=0.
\]
Under the present specialization,
\[
y_0(n)
=-x^2
\frac{(x^{2n+2}-1)(x^{2n+2}+1)(x^{2n+4}-1)}
{(x^{2n+3}-1)(x^{2n+3}+1)(x^{2n+5}+1)}
\cdot
\frac{x^{4n+8}+x^{2n+4}-x+1}
{x^{4n+4}+x^{2n+2}-x+1},
\]
and in particular $y_0(n)\neq 0$ for every integer $|x|\geq2$.
Indeed, none of the factors $x^r\pm1$ in this formula vanishes, and
\[
x^{4n+4}+x^{2n+2}-x+1>0,
\qquad
x^{4n+8}+x^{2n+4}-x+1>0.
\]
Moreover,
\[
|y_0(n)|=\rho^3\left(1+O_\rho(\rho^{-2n})\right).
\]
Set $W_n:=a_nb_{n+1}-a_{n+1}b_n$.
By \cref{lem:casoratian},
$W_{n+1}=y_0(n)W_n$.
The initial value is
\[
W_0=-a_1
=-\frac{x^2(x^4+x^2-x+1)}
{(x^2-x+1)(x+1)}
\neq0.
\]
Hence, for every $n$,
\[
W_n\neq 0,\qquad
|W_n|=\rho^{3n+O_\rho(1)}.
\]
Therefore, for all sufficiently large $n$,
\[
0<\left|\frac{a_n}{b_n}-\frac{a_{n+1}}{b_{n+1}}\right|
=\frac{|W_n|}{b_nb_{n+1}}
=\rho^{-6n^2-5n+O_\rho(1)}.
\]
Since we have already shown that $a_n/b_n\to S_x$,
\[
\frac{a_n}{b_n}-S_x
=\sum_{r=n}^{\infty}
\left(
\frac{a_r}{b_r}-\frac{a_{r+1}}{b_{r+1}}
\right).
\]
Put $d_r:=a_r/b_r-a_{r+1}/b_{r+1}$.
Then $|d_{r+1}/d_r|=\rho^{-12r+O_\rho(1)}$.
Thus, for all sufficiently large $n$, we have $\sum_{r>n}|d_r|<|d_n|/2$, and the first term dominates the sum of the absolute values of the remaining terms.
It follows that
\[
0<\left|\frac{a_n}{b_n}-S_x\right|
=\rho^{-6n^2-5n+O_\rho(1)}.
\]
Combining this with the estimate for $b_n$, we obtain
\[
0<|b_nS_x-a_n|
=\rho^{-3n^2-4n+O_\rho(1)}.
\]

\subsection{Parameter symmetry}

Consider the parameter transformation for $\beta=q$ mentioned in Section~2,
\[
(\alpha,\gamma)
\longmapsto
(\alpha^*,\gamma^*)
=\left(\frac{q^2}{\alpha},\frac{\gamma q}{\alpha}\right).
\]
Under the present specialization, this becomes
\[
(\alpha,\gamma)=(-X,X)
\longmapsto
(\alpha^*,\gamma^*)=(-X^3,-X^2).
\]
We denote the objects associated with the parameters $(\alpha^*,q,\gamma^*)$ by a superscript star.
For example,
\[
\mathsf b^*(n,k)
=(-1)^kX^{k^2}
\binom nk_{X^2}
\frac{(X^2;X^2)_{n+k}}
{(-X;X^2)_k(-X^3;X^2)_n}.
\]
By (2.1),
\[
\mathsf b_n+(1+X^{-1})\mathsf a_n
=\frac{(-X^2;X^2)_n}{(X;X^2)_n}\mathsf b_n^*.
\]

\subsection{A denominator-clearing factor}

For $n\geq0$, set
\[
\mathsf I_n
:=\frac{(X;X^2)_n(-X;X^2)_{n+1}}
{(X^2;X^2)_n}.
\]
We write $I_n$ for the value of $\mathsf I_n$ at $X=x$.
The degree of $\mathsf I_n$ in $X$ is
\[
\deg \mathsf I_n=n^2+(n+1)^2-n(n+1)=n^2+n+1.
\]
Consequently, for a fixed integer $x$ with $|x|\geq 2$,
\[
I_n\neq 0,\qquad |I_n|=\rho^{n^2+n+1+O_\rho(1)}.
\]

\begin{proposition}\label{prop:b4-integrality}
We have $\mathsf I_n\mathsf a_n,\mathsf I_n\mathsf b_n\in\mathbb Z[X]$.
\end{proposition}

Let us prove \cref{prop:b4-integrality}.
First, using the identity
$(-X^3;X^2)_k=(-X;X^2)_{k+1}/(1+X)$,
we obtain
\[
\mathsf I_n\mathsf b(n,k)=(1+X)\mathsf U_{n,k},
\tag{6.1}
\]
\[
\mathsf I_n\frac{(-X^2;X^2)_n}{(X;X^2)_n}\mathsf b^*(n,k)
=(1+X)\mathsf V_{n,k},
\tag{6.2}
\]
where
\[
\mathsf U_{n,k}
:=(1+X^{2n+1})X^{k(k+1)}
\frac{(X;X^2)_n(X^2;X^2)_{n+k}}
{(X^2;X^2)_k(X^2;X^2)_{n-k}(-X;X^2)_{k+1}},
\]
\[
\mathsf V_{n,k}
:=(-1)^kX^{k^2}
\frac{(-X^2;X^2)_n(X^2;X^2)_{n+k}}
{(X^2;X^2)_k(X^2;X^2)_{n-k}(-X;X^2)_k}.
\]
\begin{lemma}\label{lem:b4-polynomials}
We have $\mathsf U_{n,k},\mathsf V_{n,k}\in\mathbb Z[X]$.
\end{lemma}

\begin{proof}
Let $\nu_d$ denote the order of $\Phi_d(X)$ in a rational function.
Also set
\[
\Delta_\ell(n,k)
:=\left\lfloor\frac{n+k}{\ell}\right\rfloor
-\left\lfloor\frac{k}{\ell}\right\rfloor
-\left\lfloor\frac{n-k}{\ell}\right\rfloor.
\]
The difference between the first and second terms is the number of multiples of $\ell$ in the interval $(k,n+k]$, which is at least $\lfloor n/\ell\rfloor$.
Hence $\Delta_\ell(n,k)\geq0$.

Every positive integer $d$ belongs to one of the following three cases: $d$ is odd, $4\mid d$, or $d=2s$ with $s$ odd.
The last case is separated because the cyclotomic polynomials occurring in the factors $1+X^{2j+1}$ of $(-X;X^2)$ are precisely of the form $\Phi_{2s}(X)$ with $s$ odd.

If $d$ is odd, then
\[
\nu_d(\mathsf U_{n,k})=\nu_d((X;X^2)_n)+\Delta_d(n,k),\qquad
\nu_d(\mathsf V_{n,k})=\Delta_d(n,k),
\]
and both are nonnegative.
If $d=4j$, then
\[
\nu_d(\mathsf U_{n,k})=\Delta_{2j}(n,k),\qquad
\nu_d(\mathsf V_{n,k})=\nu_{4j}((-X^2;X^2)_n)+\Delta_{2j}(n,k),
\]
and again both are nonnegative.

It remains to consider indices of the form $d=2s$ with $s$ odd.
Write the residues as
\[
n\equiv r\pmod s,
\qquad
k\equiv t\pmod s,
\qquad
0\leq r,t<s,
\qquad
h=\frac{s-1}{2}.
\]
The polynomial $\Phi_{2s}(X)$ occurs in $1+X^{2j+1}$ precisely when $j\equiv h\pmod s$.
Counting the factors gives
\[
\begin{aligned}
\nu_{2s}(\mathsf U_{n,k})
&=\mathbf1_{r+t\geq s}+\mathbf1_{r<t}
-\mathbf1_{t\geq h}+\mathbf1_{r=h},\\
\nu_{2s}(\mathsf V_{n,k})
&=\mathbf1_{r+t\geq s}+\mathbf1_{r<t}
-\mathbf1_{t\geq h+1},
\end{aligned}
\]
where $\mathbf1_E$ is the indicator of the condition $E$.
The first expression can be negative only if $t\geq h$, $r\geq t$, and $r+t<s$.
These three conditions force $r=t=h$, and the last term then cancels the negative term.
For the second expression, if $t\geq h+1$, then either $r<t$ or $r\geq t$; in the latter case, $r+t\geq2t>s$.
Thus this expression is also nonnegative.

We have shown that the cyclotomic exponent is nonnegative for every $d$.
Since the numerators and denominators of $\mathsf U_{n,k}$ and $\mathsf V_{n,k}$ are products of monomials and cyclotomic polynomials, both belong to $\mathbb Z[X]$.
\end{proof}

Summing (6.1) and (6.2) over $k$ gives
\[
\begin{aligned}
\mathsf I_n\mathsf b_n&=(1+X)\sum_{k=0}^{n}\mathsf U_{n,k}
\in(1+X)\mathbb Z[X],\\
\mathsf I_n\bigl(\mathsf b_n+(1+X^{-1})\mathsf a_n\bigr)
&=(1+X)\sum_{k=0}^{n}\mathsf V_{n,k}
\in(1+X)\mathbb Z[X].
\end{aligned}
\]
Taking the difference of these two identities gives
\[
(1+X^{-1})\mathsf I_n\mathsf a_n\in(1+X)\mathbb Z[X].
\]
Since $1+X^{-1}=(1+X)/X$,
$\mathsf I_n\mathsf a_n\in X\mathbb Z[X]\subset\mathbb Z[X]$.
This proves the proposition.

\subsection{Irrationality and irrationality measure}

For an integer $|x|\geq2$, set
\[
A_n:=I_na_n,
\qquad
B_n:=I_nb_n.
\]
The polynomiality proved above gives $A_n,B_n\in\mathbb Z$.
Combining the preceding estimates, we obtain
\[
|A_n|+|B_n|
=\rho^{4n^2+o(n^2)},
\]
\[
0<|B_nS_x-A_n|
=\rho^{-2n^2+o(n^2)}.
\]
Moreover, $B_n\neq0$ for all sufficiently large $n$, and
\[
A_nB_{n+1}-A_{n+1}B_n
=I_nI_{n+1}W_n
\neq0.
\]
Here we used $I_n\neq0$ and $b_n\neq0$ for all sufficiently large $n$.
Applying \cref{lem:irrationality-measure} with
$\kappa=4$ and $\lambda=2$,
we conclude that $S_x$ is irrational and that
\[
\mu(B_4(x^{-1}))=\mu(S_x)\leq1+\frac42=3.
\]
This proves \cref{thm:b4}.

\section{Identification with the Coussement--Smet Pad\'e approximants}\label{sec:coussement-smet}

Let $q_1\in(0,1)\cap\mathbb Q$ and
$q_2=p_2^{-1}$ with $p_2\in\mathbb Z$, $p_2\geq 2$.
Coussement--Smet~\cite{coussementsmet2009irrationality} constructed Pad\'e
approximants to the series
\[
h^\pm(q_1,q_2)
:=\sum_{m=1}^\infty\frac{q_1^m}{1\pm q_2^m}.
\]
When $q_1$ and $q_2$ have a common integral base, that is, when they can be
written as
\[
q_1=p^{-r_1},
\qquad
q_2=p^{-r_2}
\]
for an integer $p\geq2$ and relatively prime positive integers $r_1,r_2$,
Coussement--Smet proved the irrationality of $h^-(q_1,q_2)$ and obtained the
upper bound $m^-(r_2)$ for its irrationality measure.

Set $p_1:=q_1^{-1}$ throughout this section.
In the general setting, $p_1>1$ is rational and $p_2\geq2$ is an integer.
Specializing the three-parameter $q$-WZ pair to
$(q,\alpha,\beta,\gamma)=(p_2,p_2,p_1,p_2)$,
the formula for the case $\alpha=q$ in Section 2 gives
\[
\begin{aligned}
\mathcal S_{p_2,p_1,p_2}(p_2)
&=\sum_{k=0}^{\infty}\frac{1}{p_1p_2^k-1}
=\sum_{k=0}^\infty\frac{q_1q_2^k}{1-q_1q_2^k}\\
&=\sum_{k=0}^\infty\sum_{m=1}^\infty q_1^mq_2^{km}
=\sum_{m=1}^\infty\frac{q_1^m}{1-q_2^m}\\
&=h^-(q_1,q_2).
\end{aligned}
\]
The interchange of the two sums follows from Tonelli's theorem, since all
terms are nonnegative.
The purpose of this section is to show that the rational approximant of Coussement--Smet with degree $n$ and evaluation point $p_2^{n+1}$ coincides with $a_n/b_n$ in our construction.
In particular, our method recovers the upper bound $m^-(r_2)$ for the irrationality measure of $h^-(p^{-r_1},p^{-r_2})$.

\subsection{The Coussement--Smet Pad\'e approximation}

Coussement--Smet define
\[
f(z)
:=\sum_{\ell=0}^{\infty}
\frac{q_1^\ell}{z-q_2^\ell}
\]
for $z\notin q_2^{\mathbb{Z}_{\geq 0}}$, and choose polynomials $P_n,Q_n$ with
$\deg Q_n=n$ and $\deg P_n\leq n-1$ such that
\[
Q_n(z)f(z)-P_n(z)=O(z^{-n-1})
\qquad(z\to\infty).
\]
We use below the same normalization as in equation (2.6) of the
original paper, namely
\[
Q_n(z)
=(-1)^n\sum_{k=0}^n(-1)^k
\frac{(p_1;p_2)_{n+k}}
{(p_1;p_2)_k(p_2;p_2)_k(p_2;p_2)_{n-k}}
p_2^{k(k-1)/2-nk}z^k.
\tag{7.1}
\]
With this normalization,
\[
Q_n(0)=(-1)^n\frac{(p_1;p_2)_n}{(p_2;p_2)_n}.
\]
The corresponding polynomial $P_n$ is
\[
P_n(z)
=\sum_{\ell=0}^{\infty}q_1^\ell
\frac{Q_n(z)-Q_n(q_2^\ell)}{z-q_2^\ell}.
\tag{7.2}
\]
Moreover, $Q_n$ satisfies the orthogonality relations
\[
\sum_{\ell=0}^{\infty}q_1^\ell
Q_n(q_2^\ell)q_2^{j\ell}=0
\qquad(0\leq j<n).
\tag{7.3}
\]

For a positive integer $N$, set
\[
H_N:=\sum_{m=1}^{N-1}\frac{q_1^m}{1-q_2^m},
\]
with the convention that an empty sum is zero.
Then equation (2.14) of the original paper becomes
\[
h^-(q_1,q_2)
=H_N+\left(\frac{p_2}{p_1}\right)^Nf(p_2^N).
\tag{7.4}
\]
Thus, before clearing denominators, the Pad\'e approximant is
\[
R_{n,N}
:=H_N+\left(\frac{p_2}{p_1}\right)^N
\frac{P_n(p_2^N)}{Q_n(p_2^N)}.
\tag{7.5}
\]

\subsection{Identification of the approximants}

\begin{proposition}\label{prop:pade-identification}
Let $n\geq0$, and set
$z_n:=p_2^{n+1}$, $\rho_n:=(p_2/p_1)^{n+1}$.
When the three-parameter construction is specialized to
$(q,\alpha,\beta,\gamma)=(p_2,p_2,p_1,p_2)$, we have
\[
Q_n(z_n)=(-1)^nb_n,
\tag{7.6}
\]
\[
a_n
=b_nH_{n+1}+(-1)^n\rho_nP_n(z_n).
\tag{7.7}
\]
In particular, $b_n\neq0$, and
\[
\frac{a_n}{b_n}=R_{n,n+1}.
\tag{7.8}
\]
\end{proposition}

\begin{proof}
For this specialization, the weight is
\[
b(n,k)
=(-1)^kp_2^{k(k+1)/2}
\frac{(p_1;p_2)_{n+k}}
{(p_1;p_2)_k(p_2;p_2)_k(p_2;p_2)_{n-k}}.
\tag{7.9}
\]
Substituting $z=z_n$ into (7.1) gives
\[
p_2^{k(k-1)/2-nk}z_n^k
=p_2^{k(k+1)/2},
\]
so comparison term by term yields (7.6).
Since $Q_n$ is an orthogonal polynomial with respect to a positive discrete
measure on $(0,1]$, all its zeros lie in $(0,1)$.
Because $z_n>1$, it follows that $Q_n(z_n)\neq0$, and hence $b_n\neq0$.

We next prove (7.7).
The specialized functions $F,G$ are
\[
F(n,k)
=-\frac{(p_1;p_2)_k}{(p_1;p_2)_{n+k+1}},
\qquad
G(n,k)
=\frac{p_2^{n+1}}{p_2^{n+1}-1}F(n,k).
\tag{7.10}
\]
The error representation obtained in Section 2 gives
\[
h^-(q_1,q_2)b_n-a_n
=\sum_{k=0}^n b(n,k)\sum_{m=n}^{\infty}G(m,k).
\tag{7.11}
\]
We compare the right-hand side with the error of the Pad\'e approximation.

Write $Q_n(z)=\sum_{k=0}^nd_{n,k}z^k$.
Comparing the individual terms in (7.6), we obtain
\[
b(n,k)=(-1)^nd_{n,k}z_n^k.
\tag{7.12}
\]
Substituting $q_1=p_1^{-1}$ and $q_2=p_2^{-1}$ into (7.10) also gives
\[
G(m,k)
=\frac{(-1)^mq_1^{m+1}
q_2^{m(m+1)/2+(m+1)k}}
{(1-q_2^{m+1})(q_1q_2^k;q_2)_{m+1}}.
\tag{7.13}
\]
By the $q$-binomial theorem,
\[
\frac{1}{(u;q_2)_{m+1}}
=\sum_{\ell=0}^{\infty}
\binom{m+\ell}{\ell}_{q_2}u^\ell
\qquad(|u|<1).
\tag{7.14}
\]
Substituting (7.12)--(7.14) into (7.11), interchanging the order of the
absolutely convergent sums, and setting $j=m-n+\ell$, we obtain
\[
h^-(q_1,q_2)b_n-a_n
=q_1^{n+1}\sum_{j=0}^{\infty}
q_1^j C_{n,j}Q_n(q_2^j),
\tag{7.15}
\]
where
\[
C_{n,j}
:=\sum_{s=0}^j(-1)^s
\frac{q_2^{(n+s)(n+s+1)/2}}
{1-q_2^{n+s+1}}
\binom{n+j}{j-s}_{q_2}.
\tag{7.16}
\]
We now rewrite this coefficient in a form suitable for applying the
orthogonality relation (7.3).
Setting $M:=n+j+1$, we have
\[
\frac{1}{1-q_2^{n+s+1}}
\binom{n+j}{j-s}_{q_2}
=\frac{(q_2;q_2)_{M-1}}
{(q_2;q_2)_{j-s}(q_2;q_2)_{n+s+1}}.
\]
Furthermore, the case $x=1$ of the finite $q$-binomial theorem gives
\[
\sum_{t=0}^{M}
\frac{(-1)^tq_2^{t(t-1)/2}}
{(q_2;q_2)_t(q_2;q_2)_{M-t}}=0.
\]
Using these identities in (7.16), we obtain
\[
C_{n,j}
=\frac{(-1)^n}{1-q_2^{n+j+1}}
+\Pi_n(q_2^j),
\tag{7.17}
\]
where
\[
\Pi_n(X)
:=(-1)^n\sum_{t=1}^n
\frac{(-1)^tq_2^{t(t-1)/2}}{(q_2;q_2)_t}
\prod_{u=n+2-t}^{n}(1-q_2^uX),
\tag{7.18}
\]
and the product for $t=1$ is understood to be $1$.
The sum in (7.18) is defined to be zero when $n=0$, while
$\deg\Pi_n\leq n-1$ for $n\geq1$.
It follows from (7.3) that
\[
\sum_{j=0}^{\infty}
q_1^j\Pi_n(q_2^j)Q_n(q_2^j)=0.
\tag{7.19}
\]
On the other hand, $z_n=q_2^{-(n+1)}$ and
$\rho_n=(q_1/q_2)^{n+1}$ imply that
\[
\frac{\rho_n}{z_n-q_2^j}
=\frac{q_1^{n+1}}{1-q_2^{n+j+1}}.
\]
Combining (7.15), (7.17), (7.19), and (7.2), we obtain
\[
h^-(q_1,q_2)b_n-a_n
=(-1)^n\rho_n
\bigl(Q_n(z_n)f(z_n)-P_n(z_n)\bigr).
\tag{7.20}
\]
Substituting (7.4) and (7.6) into (7.20) yields (7.7), and (7.5) then gives
(7.8).
\end{proof}

\subsection{Relation to the arithmetic results}

The identification in \cref{prop:pade-identification} concerns the rational
approximants before denominators are cleared.
Write the numerator and denominator used by Coussement--Smet before clearing
denominators as
\[
\begin{aligned}
\widehat A_{n,N}
&:=p_1^NQ_n(p_2^N)H_N+p_2^NP_n(p_2^N),\\
\widehat B_{n,N}
&:=p_1^NQ_n(p_2^N).
\end{aligned}
\]
Then (7.6) and (7.7) give
\[
\widehat A_{n,n+1}=(-1)^np_1^{n+1}a_n,
\qquad
\widehat B_{n,n+1}=(-1)^np_1^{n+1}b_n.
\tag{7.21}
\]
Consequently, multiplying both sides of (7.21) by the factor used by
Coussement--Smet to clear denominators yields the integer linear forms used
in their irrationality proof.

For the representation in terms of a reduced common base
\[
q_1=p^{-r_1},
\qquad
q_2=p^{-r_2},
\qquad
p\geq2,
\qquad
\gcd(r_1,r_2)=1,
\]
set
\[
g_r
:=\frac{3}{\pi^2}
\left[
1+
\prod_{\substack{\varpi\mid r\\\varpi\ \mathrm{prime}}}
\frac{\varpi^2}{\varpi^2-1}
\sum_{\substack{1\leq j<r\\(j,r)=1}}
\left(\frac1{j^2}-\frac1{r^2}\right)
\right]
\]
and
\[
m^-(r)
:=1+\frac{3/2+g_r}{3/2-g_r}
=\frac{6}{3-2g_r}.
\tag{7.22}
\]
Choosing the evaluation indices so that $N/n\to1$ and then applying equations
(1.5), (1.8), (1.9), and (1.14) of Coussement--Smet yields precisely these
quantities $g_r$ and $m^-(r)$.
Multiplying the numerator and denominator of the approximant in
\cref{prop:pade-identification} by the factor used by Coussement--Smet and
applying their asymptotic estimates gives
\[
h^-(p^{-r_1},p^{-r_2})\notin\mathbb Q,
\qquad
\mu\!\left(h^-(p^{-r_1},p^{-r_2})\right)
\leq m^-(r_2).
\]
If one starts from a representation with $\gcd(r_1,r_2)>1$, one first
replaces the base so that the exponents become relatively prime and then
uses the resulting second exponent in (7.22).

\appendix
\crefalias{section}{appendix}

\section{Maple verification of the common recurrence}\label{app:maple}

The symbolic computation used in the proof of
\cref{prop:common-recurrence} was performed with Maple 2026.0 and the
standard package \texttt{QDifferenceEquations}.  The complete input and
saved output are archived as \texttt{verify\_master\_common\_recurrence.mpl}
and \texttt{verify\_master\_common\_recurrence.out}, respectively, at
\href{https://doi.org/10.5281/zenodo.21950343}{Zenodo,
doi:10.5281/zenodo.21950343}.

The script applies \texttt{QDifferenceEquations:-Zeilberger} to the weight
$b(n,k)$ and verifies that the normalized operator has degree two and leading
coefficient one, that its constant coefficient is the displayed function
$y_0(n)$, and that the resulting Z-pair identity holds.  It then applies the
same algorithm to the term $H(n,k)$ occurring in the proof and verifies the
corresponding order-zero identity.  Finally, it checks the four boundary
values at $k=0,n+3$ and evaluates the finite-product form of the second
identity at two independent sets of rational parameters.  The final marker
\texttt{MASTER\_WEIGHT\_HAS\_COMMON\_RECURRENCE} is printed only when all
these checks succeed.

The symbolic identities are identities of rational functions in
$q,\alpha,\beta,\gamma$.  If a specialization makes a denominator in the
normalized operator vanish, one may instead use the unnormalized operator or
clear the denominators of the recurrence.

\section{Maple verification of reciprocal-product duality}\label{app:maple-duality}

The computation associated with
\cref{sec:reciprocal-product-duality} is archived as
\texttt{verify\_reciprocal\_product\_duality.mpl} and
\texttt{verify\_reciprocal\_product\_duality.out} at the same
\href{https://doi.org/10.5281/zenodo.21950343}{Zenodo record,
doi:10.5281/zenodo.21950343}.

After specializing $\beta=q$, the script generates the recurrence
coefficients $y_0,y_1$, substitutes
$(\alpha^*,\gamma^*)=(q^2/\alpha,\gamma q/\alpha)$, and verifies the two
gauge-covariance identities for the coefficients.  It also verifies the four
initial-value identities obtained from
\[
\frac{\alpha}{q}a_j=R_ja_j^*,
\qquad
b_j+\left(1-\frac{\alpha}{q}\right)a_j=R_jb_j^*
\quad (j=0,1).
\]
All six differences simplify to zero as rational functions in
$q,\alpha,\gamma$; the saved output records \texttt{true} for each check.

\printbibliography[title={References}]

\end{document}